\documentclass[letterpaper,twoside]{article}
\usepackage{fancyhdr}
\usepackage[left=3.3cm, top=3.2cm, right=3.3cm, bottom=3.2cm, marginparwidth=2.6cm, marginparsep=5mm]{geometry}
\usepackage{tikz-cd}
\usetikzlibrary{matrix}
\usepackage{amsmath} 
\usepackage{enumerate}
\usepackage{amsfonts}
\usepackage{amssymb}
\usepackage{mathtools}
\usepackage{mathrsfs}
\usepackage{xcolor}
\usepackage{array}
\usepackage{blkarray}
\usepackage{todonotes}
\usepackage{xparse}
\usepackage{amsthm}
\usepackage{indentfirst}
\usepackage{float}
\usepackage{booktabs}
\usepackage[perpage]{footmisc}
\usepackage{titlesec} 
\usepackage{ stmaryrd }
\usepackage{charter}
\usepackage{tgheros}
\usepackage{bbm}
\usepackage[T1]{fontenc}
\usepackage[expert]{mathdesign}
\usepackage[OT2,T1]{fontenc}
\DeclareSymbolFont{cyrletters}{OT2}{wncyr}{m}{n}
\DeclareMathSymbol{\Sha}{\mathalpha}{cyrletters}{"58}

\titleformat{\section}[hang]{
    \usefont{T1}{qhv}{b}{n}\selectfont} %
    {} 
    {0em}
    {\hspace{-0.4pt}\Large \thesection\hspace{0.6em}}
\titleformat{name=\section,numberless}[display]
{
	\usefont{T1}{qhv}{b}{n}\selectfont} %
{} 
{0em}
{\Large}

\renewcommand{\refname}{References}

\makeatletter \g@addto@macro\@floatboxreset\centering 
\makeatother
\renewcommand{\sectionmark}[1]%
{\markright{{\thesection\ #1}}}
\renewcommand{\subsectionmark}[1]%
{}

\newcommand{\helv}{%
\fontfamily{phv}\fontseries{b}\fontsize{9}{11}\selectfont}
\usepackage[ocgcolorlinks]{hyperref}
\definecolor{linkcolour}{rgb}{0,0.2,0.6}
\hypersetup{colorlinks,breaklinks,urlcolor=linkcolour, linkcolor=linkcolour}
\DeclareDocumentCommand{\newfaktor}{s m O{0.5} m O{-0.5}}{%
  \setbox0=\hbox{\ensuremath{#2}}%
  \setbox1=\hbox{\ensuremath{\diagup}}%
  \setbox2=\hbox{\ensuremath{#4}}%
  \raisebox{#3\ht1}{\usebox0}%
  \mkern-5mu\ifthenelse{\equal{#1}{\BooleanTrue}}%
    {\diagup}%
    {\rotatebox{-44}{\rule[#5\ht2]{0.4pt}{-#5\ht2+#3\ht0+\ht0}}}%
  \mkern-4mu%
  \raisebox{#5\ht2}{\usebox2}%
}

\usepackage{xpatch}
\makeatletter
\xapptocmd{\@sect}{\csname #1mark\endcsname{#7}}{}{}
\makeatother

\newtheoremstyle{mystyle}%
  {0.4cm}%
  {0.3cm}%
  {\itshape}%
  {}%
  {\fontsize{10.5pt}{0pt}\scshape  }%
  {}%
  {  }%
  {}%

\theoremstyle{mystyle}
\newtheorem{lemma}{Lemma}[section]
\newtheorem{thm}[lemma]{Theorem}
\newtheorem{corol}[lemma]{Corollary}

\newtheorem{prop}[lemma]{Proposition}
\newtheorem*{claim}{Claim}

\newtheoremstyle{roman}%
  {0.4cm}%
  {0.3cm}%
  {\normalfont}%
  {}%
  {\scshape }%
  {}%
  {  }%
  {}%
\theoremstyle{roman}
\newtheorem{mydef}[lemma]{Definition}
\newtheorem{example}[lemma]{Example}
\newtheorem{notat}[lemma]{Notation}
\newtheorem{constr}[lemma]{Construction}
\newtheorem{remark}[lemma]{Remark}
\newtheorem{question}[lemma]{Question}

\newtheorem*{thm*}{Theorem}

\newcounter{saveenumerate}
\makeatletter
\newcommand{\enumeratext}[1]{%
\setcounter{saveenumerate}{\value{enum\romannumeral\the\@enumdepth}}
\end{enumerate}
#1
\begin{enumerate}[i)]
\setcounter{enum\romannumeral\the\@enumdepth}{\value{saveenumerate}}%
}
\makeatother

\include{functions}
\newcommand{\hyp}{\operatorname{hyp}}
\newcommand{\nchyp}{\operatorname{nchyp}}
\newcommand{\perm}{{\rm perm}}
\renewcommand{\triv}{{\rm triv}}
\newcommand{\BR}{\operatorname{BR}}
\renewcommand{\char}{{\rm char}}
\newcommand{\defl}{{ \rm defl}}
\renewcommand{\inf}{{ \rm inf}}
\renewcommand{\cyc}{{ \rm cyc}}
\newcommand{\sd}{{\mathrm{sd}}}
\edef\restoreparindent{\parindent=\the\parindent\relax}
\usepackage{parskip}
\AtBeginEnvironment{figure}{\vskip-0.5ex}
\AfterEndEnvironment{figure}{\vskip-1.5ex}
\restoreparindent

\renewcommand\up{\mathord{\uparrow}}
\renewcommand\down{\mathord{\downarrow}}
\renewcommand{\BR}{\mathord{B \hspace{-0.02cm} R}}
\newcommand{\br}{\textit{br}}
\begin{document}
\title{\vspace{-2cm} Regulator constants of integral representations of \linebreak finite groups}
\author{ Alex Torzewski}
\date{}
\thispagestyle{plain}
\maketitle
\vspace{-0.3cm}
\begin{abstract}
	Let $G$ be a finite group and $p$ be a prime.  We investigate isomorphism invariants of $\mathbb{Z}_{p}[G]$-lattices whose extension of scalars to $\q_p$ is self-dual, called regulator constants. These were originally introduced by Dokchitser--Dokchitser in the context of elliptic curves. Regulator constants canonically yield a pairing between the space of Brauer relations for $G$ and the subspace of the representation ring for which regulator constants are defined. For all $G$, we show that this pairing is never identically zero. For formal reasons, this pairing will, in general, have non-trivial kernel. But, if $G$ has cyclic Sylow $p$-subgroups and we restrict to considering permutation lattices, then we show that the pairing is non-degenerate modulo the formal kernel. Using this we can show that, for certain groups, including dihedral groups of order $2p$ for $p$ odd, the isomorphism class of any $\z_p[G]$-lattice whose extension of scalars to $\q_p$ is self-dual, is determined by its regulator constants, its extension of scalars to $\q_p$, and a cohomological invariant of Yakovlev.
\end{abstract}
\setcounter{tocdepth}{2}
\vspace{-1cm}
\tableofcontents
 
\section{Introduction}
Regulator constants are invariants of representations of a finite group introduced by Dokchitser--Dokchitser (see for example \cite{DDRegConstParity}). We shall briefly recall some of the basic properties of regulator constants (cf.\ Section \ref{regcon}).

Let $G$ be a finite group. A Brauer relation in characteristic zero (resp.\ characteristic $p$) consists of a pair of $G$-sets for which the associated permutation modules over $\q$ (resp.\ $\F_p$) are isomorphic. Characteristic zero (resp.\ $p$) relations form a free abelian group of finite rank, which we denote by $\br_0(G)$ (resp.\ $\br_p(G)$). All characteristic $p$ relations are also characteristic zero relations so that $\br_p(G)\subseteq\br_0(G)$ (Lemma \ref{BRp=Brzpapp}).

If $\R$ is a ring, then by an $\R[G]$-lattice we mean an $\R[G]$-module which is free of finite rank as an $\R$-module. We will be mainly interested in the case of $\R=\z_p$, the $p$-adic integers, and $\R=\zp$, the localisation of $\z$ at $p$. Let $\q_p$ denote the field of $p$-adic numbers.  We call a $\z_p[G]$-lattice \emph{rationally self-dual} if its extension of scalars to $\q_p$ is self-dual. Each characteristic zero Brauer relation $\theta$ defines a regulator constant $C_\theta(-)$ which assigns to a rationally self-dual $\z_p[G]$-lattice $M$ an element $C_\theta(M)\in \q_p\tii/(\z_p\tii)^2$. As will be made precise later, $C_\theta(M)$ measures the relative covolumes of certain fixed subspaces corresponding to the $G$-sets defining $\theta$ (see Definition \ref{regdef}).
 
 In several number theoretic contexts, regulator constants have been found to both coincide with naturally occurring objects and to be computationally accessible. For example, if $K/\q$ is a Galois extension of number fields with $G=\gal(K/\q)$, $E/\q$ is an elliptic curve and $M=E(K)/E(K)_\tors$, the torsion-free quotient of the Mordell-Weil group of $E$, then the regulator constants of $M\ot \z_p$ are closely related to the elliptic regulator of $E$ \cite{DDRegConstParity}. Similarly, if $M=\ok\tii/\mu_K$ is the unit group of $K$ modulo roots of unity, then the regulator constants of $M\ot \z_p$ are closely related to Dirichlet's unit group regulator \cite{BartelDih}.
 
 The applications of regulator constants are dependent on showing that regulator constants are good invariants of lattices. In this paper, we systematically investigate the strength of regulator constants as invariants of lattices.
 
 Let $a(\z_p[G])$ denote the representation ring of $G$ over $\z_p$. We write $a(\z_p[G],\sd)$ for the subring generated by $\z_p[G]$-lattices which are rationally self-dual. Set $A(\z_p[G])=a(\z_p[G])\ot_\z \q$, $A(\z_p[G],\sd)=a(\z_p[G],\sd)\ot_\z \q$, $\BR_0(G)=\br_0(G)\ot_\z \q$, $\BR_p(G)=\br_p(G)\ot_\z \q$.  Regulator constants are multiplicative in direct sums of lattices and also under summing Brauer relations. As such, if $v_p(-)$ denotes the $p$-adic valuation, then there is a pairing
 \begin{align*}
 v_p(C_{(-)}(-)) \colon \BR_0(G) \ti A(\z_p[G],\sd)&\longrightarrow \q\\
 (\theta, M) \quad\quad\qquad& \longmapsto v_p(C_\theta(M)).
 \end{align*}
The space $\BR_0(G)$ is always finite dimensional, whilst $A(\z_p[G],\sd)$ will regularly be infinite dimensional. For trivial reasons, elements of $\BR_p(G)$ always lie in the kernel of $v_p(C_{(-)}(-))$.  But, one might say that regulator constants are good invariants if the left kernel consists only of characteristic $p$ relations. It is one of our main results that if $G$ has cyclic Sylow $p$-subgroups, then this is always the case. But let us be more precise.

 Outside of small families of groups we do not have classifications of the $\z_p[G]$-lattices. On the other hand, the isomorphism classes of permutation modules, that is $\z_p[G]$-lattices on which $G$ acts by permuting a choice of basis, are easy to enumerate and it is possible to give a formula for their regulator constants in terms of group theoretic information. For this reason we shall primarily restrict our attention from $A(\z_p[G],\sd)$ to $A(\z_p[G],\perm)$, the subspace generated by permutation modules. 

Again, for trivial reasons $A(\z_p[G],\cyc)$, the subspace generated by the permutation modules $\z_p[G/H]$ for $H\le G$ cyclic, lies in the kernel of $v_p(C_{(-)}(-))$. We refer to the resulting pairing
	\[ \la \ \, , \ \ra_\perm \colon \BR_0(G)/\BR_p(G) \ti A(\z_p[G],\perm)/A(\z_p[G],\cyc) \to \q   \]
as the \emph{permutation pairing}. A not immediately obvious fact is that both $\BR_0(G)/\BR_p(G)$ and $A(\z_p[G],\perm)/A(\z_p[G],\cyc)$ are canonically isomorphic to the free $\q$-vector space on the set of conjugacy classes of $p$-hypo-elementary subgroups. With respect to this identification, the pairing is symmetric. Prior to quotienting the spaces need not have the same dimension and there is no such identification (cf.\ Remark \ref{kernels}).
\begin{thm}\label{intro1}
	For a finite group $G$ and prime $p$ such that $G$ has cyclic Sylow $p$-subgroups, the permutation pairing is non-degenerate.
\end{thm}

A formal consequence of Theorem \ref{intro1} is that the isomorphism class of a permutation module $M$ over $\z_p$ is determined by its regulator constants and the isomorphism class of $M\ot \q_p$. 

To show the theorem, we first reduce to $p$-hypo-elementary subgroups. All $p$-hypo-elementary groups with cyclic Sylow $p$-subgroups are of the form $C_{p^k}\rti C_n$ with $(p,n)=1$. In this case, we find we are able to completely explicate the matrix representing the pairing, and showing invertibility becomes a combinatorial problem (Lemma \ref{gcdmatinvertible}).

For general $G$, the permutation pairing may be degenerate. For example, when $p=3$, the group $C_3\ti C_3\ti S_3$ has a Brauer relation $\theta_\Sigma$ whose regulator constant is trivial on all permutation modules (see Section \ref{C3C3S3}). I do not know if there are other lattices for which $C_{\theta_\Sigma}(-)$ does not vanish.

We do however provide a partial result for arbitrary $G$. For any group $G$, there is a canonical Brauer relation with leading term $[G]$ called the \emph{Artin relation}, which we denote by $\theta_G$ (see Definition \ref{artinrelation}). Let $\one_G$ denote the trivial $\z_p[G]$-module.
\begin{thm}\label{intro2}
	For any finite group $G$ and prime $p$, we have $v_p(C_{\theta_G}(\one_G))\neq 0$.
\end{thm}
The proof is group theoretic in nature and completely independent of that of Theorem \ref{intro1}.

The final aspect of the paper is an extended application of Theorem \ref{intro1}. Suppose now that $G$ has a cyclic Sylow $p$-subgroup $P$, and write $P_i\le P$ for the subgroup of order $p^i$. For a $\z_p[G]$-lattice $M$, Yakovlev \cite[Thm.\ 2.1]{Yak} showed that the diagram \vspace{-5pt}
\begin{figure}[H]
	\begin{tikzcd}[column sep=small]&&&&&
		H^{1}(P_r,M) \ar[shift left=0.6ex]{r}& H^{1}(P_{r-1},M)\ar[shift left=0.6ex]{r}\ar[shift left=0.6ex]{l} &\ar[shift left=0.6ex]{l} \quad ... \quad \ar[shift left=0.6ex]{r} &  H^{1}(P_{0},M)\ar[shift left=0.6ex]{l}&&&(\star)
	\end{tikzcd}
	\vspace{-10pt}
\end{figure}\noindent
determines the isomorphism class of $M$ up to summands which are \emph{trivial source}, i.e.\ summands of permutation modules. Here, the horizontal maps are restriction and corestriction, and each cohomology group is considered as an $N_G(P_i)$-module (cf.\ Section \ref{yakres}). Thus, $M$ would be completely determined if one could provide invariants which constrain the remaining trivial source summand. We refer to $(\star)$ as the \emph{Yakovlev diagram} of $M$.

Our main result gives conditions for when Theorem \ref{intro1} can be used to determine the remaining trivial source summand of $M$. Denote by $A(\zp[G],\triv)$ the subring of the representation ring $A(\zp[G])$ generated by trivial source lattices (see Definition \ref{trivsource}). Note that extension of scalars defines an inclusion $A(\zp[G])\hookrightarrow A(\z_p[G])$ \cite[Thm.\ 5.6 {\it iii)}]{ReinerSurvey} and so an isomorphism of the subrings generated by permutation modules $A(\zp[G],\perm)=A(\z_p[G],\perm)$. The condition is then that $A(\zp[G],\triv)=A(\z_p[G],\perm)$.
\begin{thm} \label{intromain} Let $G$ be a finite group and $p$ a prime such that $G$ has cyclic Sylow $p$-subgroups and such that $A(\z_p[G],\perm)=A(\zp[G],\triv)$. Given two rationally self-dual $\z_p[G]$-lattices $M,N$, then $M\cong N$ if and only if all the following conditions hold:
	\begin{enumerate}[i)]
		\item $M\ot \q_p\cong N \ot \q_p$,
		\item for an explicit finite list of characteristic zero Brauer relations, the corresponding regulator constants of $M,N$ are equal,
				\item $M,N$ have isomorphic Yakovlev diagrams.
	\end{enumerate}
\end{thm}
This is stated precisely as Theorem \ref{main}. It is relatively straightforward to obtain extensions of this result to arbitrary lattices (cf.\ Remark \ref{arbitlat}). If the conditions of the theorem hold for all $p$-hypo-elementary subgroups of $G$, then they hold for $G$. We also provide some explicit criteria for the condition $A(\z_p[G],\perm)=A(\z_p[G],\triv)$ to be satisfied. Groups that satisfy the conditions include dihedral groups, abelian groups with cyclic Sylow $p$-subgroups and groups of order coprime to $(p-1)$.

 For some groups, such as dihedral groups $D_{2p}$ for primes $p\le 67$, the isomorphism class of a $\z[G]$-lattice $M$ is determined by its localisations at the primes dividing $|G|$. As a result, Theorem \ref{intromain} may be applied at each prime to give data which determines the isomorphism class of $M$ as a $\z[G]$-lattice (cf.\ Remark \ref{67}).
 
 It is possible to define regulator constants $C_\theta(M)$ of a $\z[G]$-lattice $M$. Then $C_\theta(M)$ is the product of the $p$-part of $C_\theta(M\ot \z_p)$ for all $p$ dividing $|G|$ (see Remark \ref{productp}). As a result, confining ourselves to $\z_p[G]$-lattices over $\z[G]$-lattices is innocuous. 
 
As part of the author's PhD thesis, we shall describe some applications of Theorem \ref{intromain} within number theory. For example, in the case of unit groups of number fields, when $p$ divides $|G|$ at most once, it is possible to reinterpret the three invariants of Theorem \ref{intromain} in terms of classical invariants of number fields  \cite[Ch.\ 3]{thesis}.

{\bf Outline:} In Section \ref{prelim}, we set out notation and recall necessary background results on Brauer relations and regulator constants. In Section \ref{paireg}, we outline precise questions on pairings arising from regulator constants. We also show that these reduce to considering $p$-hypo-elementary groups and that whenever the permutation pairing is non-degenerate, then permutation modules are determined by regulator constants and extension of scalars. In Section \ref{regconstinv}, we prove Theorem \ref{intro1}, and in Section \ref{nonvanartin}, we prove Theorem \ref{intro2}. In Section \ref{arblat}, we apply Theorem \ref{intro1} to prove Theorem \ref{intromain} on determining lattices up to isomorphism. There we also provide criteria for groups to satisfy the conditions of Theorem \ref{intromain}. Finally, in Section \ref{examples}, we provide examples and non-examples illustrating our results. 

After reading Sections \ref{prelim} and \ref{paireg}, the following Sections \ref{regconstinv} and \ref{nonvanartin} may be read completely separately from each other, as may Section \ref{arblat}, which only requires the statement of Theorem \ref{intro1}. 

\textbf{Funding:} The author was supported by a PhD studentship from the Engineering and Physical Sciences Research Council.

\textbf{Acknowledgements:} I'd like to thank Vladimir Dokchitser, David Loeffler and Chris Wuthrich for many helpful discussions and suggestions. I am also very grateful to be have had access to Magma functions written by Tim Dokchitser for calculating regulator constants. I am highly indebted to Henri Johnston for pointing out an earlier error and an anonymous referee for their very thorough reading and guidance as to how best to state the results. Finally, special thanks to Matthew Spencer and Alex Bartel without whom this paper would not have been possible.
\section{Preliminaries}\label{prelim}
\subsection{Notation}\label{trisou}
	Throughout, $G$ shall denote a finite group, $p$ a prime and $\R$ any ring, but we will be most concerned with $\R=\F_p,\zp,\z_p,\q$ or $\q_p$.
\begin{notat}\label{notat}
	We fix the following notation:
	\begin{itemize}
		\item Let $\one_G$ or $\one$ denote the trivial $\R[G]$-module. Where the choice of ring requires emphasis we write $\one_{\R,G}$.
		\item Given a subgroup $H\le G$ and an $\R[G]$-module $M$, we shall denote the restriction of $M$ to $H$ by $M\down ^G_H$ or $M\down_H$. Similarly, given an $\R[H]$-module $N$, we write $N\up^G_H$ or $N\up^G$ for the induction of $N$ to $G$.
		\item We say that an $\R[G]$-module is an \emph{$\R[G]$-lattice} if it is free of finite rank as an $\R$-module. Let $a(\R[G])$ denote the \emph{representation ring} of $G$. As an abelian group, $a(\R[G])$ consists of formal $\z$-linear combinations of isomorphism classes of $\R[G]$-lattices, subject to relations of the form $[M]+[N]=[M\op N]$. Here we use $[M]$, or just $M$, to denote the element of $a(\R[G])$ corresponding to an $\R[G]$-lattice $M$. The ring structure on $a(\R[G])$ is given by setting $[M]\cdot [N]=[M\ot_\R N]$. Induction defines a group homomorphism $\ind\colon a(\R[H])\to a(\R[G])$, whilst restriction defines a ring homomorphism $\res \colon a(\R[G])\to a(\R[H])$.
		\item Let $A(\R[G])$ denote $a(\R[G])\ot \q$. All our main results do not require the integral structure. As a result we frequently deal only with $A(\R[G])$ even though some intermediate results also hold integrally.
		\item Recall that a permutation module is a finite direct sum of modules of the form $\one \up^G_H$ as $H$ runs over subgroups of $G$. We denote the subgroup of $a(\R[G])$ spanned by such modules by $a(\R[G],\perm)$. The equality $\one\up^G_H\ot_\R \one \up^G_K = \one_H \up^G\down_K\up^G$ and Mackey's formula show that $a(\R[G],\perm)\subseteq a(\R[G])$ is a subring, which we call the \emph{permutation ring}. Both $\res,\ind$ restrict to maps of permutation rings, the former due to Mackey's formula. We set $A(\R[G],\perm)=a(\R[G],\perm)\ot\q$.
		\item Let $A(\R[G],\cyc)$ be the $\q$-subalgebra spanned by $\one \up^G_H$ as $H$ runs only over cyclic subgroups.
		\item Given a quotient $q\colon G\to G/N$ and an $\R[G/N]$-module $M$, we denote by $\inf^G_{G/N}(M)$ the inflation of $M$ to $G$. This defines a ring homomorphism $\inf \colon a(\R[G/N])\to a(\R[G])$, which restricts to a map of permutation rings since, for $H\le G/N$, $\inf^G_{G/N}(\one_H\up^{G/N})= \one \up^G_{q^{-1}(H)}$.
		\item In the same notation, given a $G$-module $M$, we define its deflation to $G/N$, $\defl^G_{G/N}M$, to be the fixed submodule $M^N$ with $G/N$-action. Restricting to permutation modules, $(\one \up^G_H)^N\cong \one \up^G_{NH} $, so that $\defl(\one\up^G_H)\cong \one \up^{G/N}_{q(H)}$. The composite $\defl \circ \inf$ is the identity on all of $a(\R[G/N])$.
				\item By $H\le_G G$, we denote a conjugacy class of subgroups of $G$ with representative $H$. When used in indices, the symbol $\le_G $ denotes indexing over conjugacy classes of subgroups. Thus, $\sum_{H\le_G G} 1$ is equal to the number of conjugacy classes of subgroups.
	\end{itemize}
\end{notat}
\begin{remark}\label{infrepring}
	Recall that a module is called \emph{indecomposable} if it can not be written as a direct sum of proper submodules. When $\R=\z_p$ or $\q$, every $\R[G]$-lattice admits a unique decomposition into direct sums of indecomposables \cite[Thm.\ 5.2]{ReinerSurvey}, so that $a(\R[G])$ is free as a $\z$-module with a basis given by isomorphism classes of indecomposable modules, but it need not have finite rank if $\R=\z_p$ and $p$ divides $|G|$ (see \cite[Sec.\ 33]{CR1}). Unique decomposition also ensures that for any two $\z_p[G]$-lattices $M,M'$, $[M]=[M']$ $\iff$ $M\cong M'$.
	
	If $\R =\z_{(p)}$, then extension of scalars defines an inclusion $a(\zp[G])\hookrightarrow a(\z_p[G])$ \cite[Thm.\ 5.6 {\it iii)}]{ReinerSurvey} so we may again detect a lattice's isomorphism class from its class in the representation ring. However, $\zp[G]$-lattices need not admit unique decomposition into indecomposables \cite{Beneish} and in general there is no obvious basis of $a(\zp[G])$. For simplicity, we shall often write $a(\zp[G])\subseteq a(\z_p[G])$.
\end{remark}

\subsection{Brauer relations}\label{brarel}

\begin{mydef}
A $G$-set is a set with a left action of $G$. We define the \emph{Burnside ring} $b(G)$ of $G$ to be the free abelian group on isomorphism classes of finite $G$-sets, quotiented by relations of the form $[X\coprod Y]-[X]-[Y]$ where $X,Y$ are any $G$-sets, and $[X], [Y]$ are the corresponding elements of the free group. The ring structure is then given by setting $[X]\cdot [Y]=[X\ti Y]$.

By decomposing $G$-sets into their orbits, we find that $b(G)$ is a free $\z$-module on isomorphism classes of transitive $G$-sets. Every transitive $G$-set is of the form $G/H$ for some subgroup $H\le G$, where $H$ is unique up to conjugacy. We denote the element of $b(G)$ corresponding to $G/H$ by $[H]$. Thus, $b(G)$ is free as a $\z$-module on the set of symbols $[H]$ for $H\le _G G$.

We write $B(G)$ for $b(G)\ot \q$.
\end{mydef}
\begin{constr}
	For any ring $\R$, a $G$-set $X$ canonically defines a permutation module and we obtain a surjective map
	\[ b_\R \colon b(G) \to a(\R[G],\perm) , \] 
	which sends $[H]$ to $\one_{\R,H}\up^G$.
\end{constr}
\begin{mydef}\label{brauerrel}
	A \emph{Brauer relation}, of a group $G$ over a ring $\R$, is an element of $\ker b_\R\subseteq b(G)$. We shall refer to the ideal $\ker b_\R$ as the space of Brauer relations over $\R$ and shall denote it by $\br_\R(G)$.
	
	When $\R=\q$ or $\F_p$, we call a Brauer relation over $\R$ a relation in characteristic zero or characteristic $p$, respectively, and denote $\br_\R(G)$ by $\br_0(G),\br_p(G)$ respectively. In the literature it is common to refer to a characteristic zero relation as simply a Brauer relation, and we shall often do the same.
\end{mydef}
\begin{example}\label{S3}
	If $G=S_3$, a characteristic zero Brauer relation is given by
	\[ \theta \colon [1]+2[G]-[C_3]-2[C_2].\]
	We shall see that in fact, $\br_0(S_3)=\theta\cdot \z$.
\end{example}
All characteristic $p$ relations are relations in characteristic zero also:
\begin{lemma}\label{BRp=Brzpapp} As subspaces of $b(G)$, $\br_p(G)=\br_{\zp}(G)=\br_{\z_p}(G)\subseteq \br_0(G)$.
\end{lemma}
\begin{proof}
	Via the factorisation $\zp \to \z_p \to \F_p$, the map $b_{\F_p}$ factors as
	\[ b(G)\to a(\zp[G],\perm)\overset{M\mapsto M \ot \z_p}{\to} a(\z_p[G],\perm)\overset{M\mapsto M \ot \F_p}{\to}a(\F_p[G],\perm).\]
	The middle map is an isomorphism by \cite[Thm 5.6 \/iii)]{ReinerSurvey}, as is the last map by \cite[3.11.4 \/i)]{Benson95},
	so the kernels of $b_{\F_p},b_{\zp}$ and $b_{\z_p}$ agree. As $b_\q$ factors through $b_{\zp}$, there is an inclusion $\br_{\zp}(G)\subseteq \br_\q(G)$.
\end{proof}
\begin{notat} Let $G$ be a finite group and $H\le G$ a subgroup.
	\begin{itemize}
		\item Given an $H$-set $X$, we let $X\up^G_H$ denote the induced $G$-set $(G\ti X)/H$; here the $H$-equivalence is by acting on $G$ on the right and $X$ on the left, whilst $G$ acts on the resulting set via its left action on $G$. For transitive $G$-sets $(G/K)$ we have $(H/K)\up^G=(G/K)$ and we shall regularly abuse notation by writing $[K]\up^G$ simply as $[K]$, where now $K$ is thought of as a subgroup of $G$.\label{Gsetinduction}
		\item If $Y$ is a $G$-set, we let $Y\down^G_H$ denote its restriction to $H$. For a subgroup $K$ of $G$, making good use of the above abuse of notation, Mackey's formula for $G$-sets states that
\begin{equation}[K]\down^G_H=\hspace{-3pt} \sum_{g \in K\backslash G /H}\hspace{-5pt}[K^g \cap H].\label{mackeyforgsets}\end{equation}
\end{itemize}
{\noindent If now $N\unlhd G$ is a normal subgroup with quotient $q\colon G \to G/N$, then}
\begin{itemize}
		\item given a $G/N$-set $X$, we denote by $\inf ^G_{G/N} X$ the inflated set $X$, on which elements of $G$ act via their image in the quotient. For $H\le G/N$, $\inf^G_{G/N}([H])= [q^{-1}(H)]$,
		\item given a $G$-set $Y$, let $\defl^G_{G/N} Y$ denote its deflation, i.e.\ the set $Y^N$ with its action of $G/N$. For a transitive $G$-set $G/H$, the fixed points under $N$ is isomorphic to $G/NH$, which as a $G/N$-set is $(G/N)/q(H)$. In other words, $\defl^G_{G/N}([H])=[q(H)]$, and thus $\defl \circ \inf$ is the identity map.
	\end{itemize}
All of these operations induce group homomorphisms on Burnside rings, but only restriction and inflation will in general be ring homomorphisms. Each of $\ind,\res, \inf,\defl$ commute with $b_\R$. As a result, each restricts to morphisms of $\br_\R(-)$.
\end{notat}
\subsection{Relations in characteristic zero} \label{relin0}
Finding an explicit basis, for an arbitrary finite group $G$, of the $\br_0(G)$ is a hard problem, which was recently completed by Bartel-Dokchitser \cite{BDbrelclassification1,BDbrelclassification2}. On the other hand, in this section we recall that, a basis of the space $\br_0(G)\ot \q$ is provided by Artin's induction theorem.
\begin{notat}
	Let $\cyc(G):=\{H \le_G G \mid H \text{--cyclic}  \}$ denote a set of representatives of each conjugacy class of cyclic subgroups.
\end{notat}
\begin{thm}[Artin's induction theorem {\cite[Thm. 2.1.3]{SnaithExplicitBrauer}}]\label{ait}
	For any finite group $G$ and $\q[G]$-module $M$,  there exists a unique $\alpha_H \in \q$ for each cyclic $H\le_GG$ such that
	\[ M = \hspace{-3pt}\sum_{H\in \cyc(G)}\hspace{-5pt} \alpha_H\one_H \up^G\in A(\q[G]) . \]
\end{thm}
\begin{mydef}
	We say that an element $\theta \in B(G)$ is supported at some set $S$ of conjugacy classes of subgroups of $G$ if the only $[H]$ with non-zero coefficients lie in $S$.
\end{mydef}
\begin{corol}\label{dimbrel0}
	For any finite group $G$,
	\begin{enumerate}[i)]
		\item the rank of $\br_0(G)$ is equal to the number of conjugacy classes of non-cyclic subgroups of $G$,
		\item there are no non-zero characteristic zero Brauer relations supported only at cyclic subgroups.
	\end{enumerate}
\end{corol}
\begin{proof}
	Immediate.
\end{proof}
 Note that a group $G$ is cyclic if and only if it has no non-trivial Brauer relations.
\begin{mydef}
	For any ring $\R$ and finite group $G$, let $b_{\R,\q}$ denote the base change of $b_\R$,
	\[  b_{\R,\q} \colon B(G) \to A(\R[G],\perm). \] 
	We shall also call an element of the kernel of $b_{\R,\q}$ a Brauer relation over $\R$ and refer to the kernel $\BR_{\R}(G):=\br_\R(G)\ot\q$ as the space of Brauer relations over $\R$. Where there is ambiguity, we shall refer to elements of the kernel of $b_{\R}$ as \emph{integral Brauer relations} and of $b_{\R,\q}$ as \emph{rational Brauer relations}.
\end{mydef}
Induction theorems of the form of Theorem \ref{ait} always give rise to a corresponding family of (possibly rational) Brauer relations.
\begin{mydef}\label{artinrelation}
	For any group $G$, let 
	\[\one_G=\sum_{H \in \cyc(G)}\hspace{-5pt}\alpha_H \one _H\up^G,\]
	where the $\alpha_H\in \q$ are given uniquely by Artin's induction theorem. Then,
	\[\theta_G=[G]-\hspace{-5pt}\sum_{H \in \cyc(G)}\hspace{-5pt}\alpha_H [H] \in B(G) \]
	is a rational Brauer relation of $G$. We call $\theta_G$ the \emph{Artin relation} of $G$. Note that if $G$ is cyclic, then $\theta_G=0\in B(G)$, otherwise $\theta_G$ is non-zero and has $[G]$-coefficient $1$. The uniqueness statement of Artin's induction theorem shows that $\theta_G$ is, up to scaling,  the unique element of $\BR_0(G)$ supported only at $G$ and cyclic subgroups.
\end{mydef} 
 The following example will be returned to in Section \ref{proofofndegcyc}.
\begin{example}\label{artinrelcyclic} Let $G$ be of the form $C_{p^r}\rti C_n$, with $p\nmid n$, and denote by $S\le C_n$ the kernel of the action $C_n \to \Aut(C_{p^r})$. Writing $s$ for $|S|$, we claim that
\[\theta_G = \frac{s}{n} \cdot[S]- [C_n]-\frac{s}{n}\cdot [C_{p^r}\ti S] + [C_{p^r}\rti C_n] ,\]
which can be checked by direct calculation. If the action of $C_n$ is not faithful, then $S$ is a non-trivial subgroup of $G$ and quotienting by $S$ results in a group of the same form but with faithful action. The Artin relation of $G$ is then the inflation of the Artin relation of $G/S$ (using that the preimage of a cyclic subgroup of $G/S$ is a cyclic subgroup of $G$).
\end{example}
	Following Notation \ref{Gsetinduction}, when it is contextually clear we are referring to $G$-relations, for a subgroup $H\le G$, we shall denote the $G$ relation $\theta_H \up^G$ simply by $\theta_H$. Artin relations are well behaved under restriction:
 \begin{lemma}\label{mackeyforartrel}
	Let $G$ be a finite group and $H,K$ subgroups. Then,
	\begin{enumerate}[i)]
		\item the restriction of the Artin relation of $G$ to $H$ is the Artin relation of $H$, i.e.\
		\[\theta_G\down_H= \theta_H,\]
		\item more generally
		\begin{equation}
		\theta_H\up^G_H\down_K=\sum_{g\in H\backslash G/K}\hspace{-5pt}\theta_{H^g\cap K}\up^K. \label{mackeyforartreleq}
		\end{equation}  
	\end{enumerate}  
\end{lemma}
\begin{proof}
	We prove {\it ii)}. Mackey's formula \eqref{mackeyforgsets} states that
	\[  [H]\down_K=\sum_{g \in H \backslash G/K}\hspace{-5pt} [H^g\cap K].  \]
	Also by Mackey, for any cyclic group $L\le H$, $[L]\up^G_H\down_K$ is supported at cyclic subgroups. But then $\theta\up^G_H\down_K$ and $\sum_{g\in H\backslash G/K}\theta_{H^g\cap K}\up^K$ are two relations whose coefficients agree at all non-cyclic subgroups and since there are no relations supported at cyclic subgroups (Corollary \ref{dimbrel0} {\it ii)}), they must therefore be equal.
\end{proof} 
\begin{lemma}\label{artinbasisofbrauerrel}
	A basis of the space of rational Brauer relations $\BR_0(G)_\q$ is given by the set $\{\theta_H\}$ of Artin relations for non-cyclic $H\le_G G$.
\end{lemma}
\begin{proof}
The $\theta_H$ are linearly independent as each is zero on non-cyclic subgroups other than $[H]$ and must span by Corollary \ref{dimbrel0} {\it i)}.
\end{proof}
\subsection{Relations in characteristic $p$}\label{relinp}

\begin{mydef}\label{phypo}
	Let $p$ be prime. A finite group $G$ is called \emph{$p$-hypo-elementary}, or simply $p$-hypo, if $G$ has a normal Sylow $p$-subgroup $P$ and $G/P$ is cyclic, i.e.\ if $G$ can be written in the form $P\rti C_n$ for $P$ a $p$-group and $(p,n)=1$.
\end{mydef}
\begin{notat}
	We denote a set $\{H\le_G G \mid  H \text{ $p$-hypo}\}$ of representatives of the conjugacy classes of $p$-hypo-elementary subgroups by $\hyp_p(G)$. Similarly, let $\nchyp_p(G):=\{H\le_G G \mid  H \text{ is $p$-hypo and non-cyclic}  \}$.
\end{notat}
Recall that characteristic $p$ relations are coincide with $\zp$ and $\z_p$-relations (Lemma \ref{BRp=Brzpapp}). Analogously to Corollary \ref{dimbrel0} we have
\begin{thm}\label{dimbrelp}
	For any finite group $G$,
	\begin{enumerate}[i)]
		\item a basis of $A(\z_p[G],\perm)$ is given by $\{\one \up^G_H\}_{H \in \hyp_p(G)}$,
		\item the rank of $\br_p(G)$ is equal to the number of conjugacy classes of non-$p$-hypo-elementary subgroups of $G$,
		\item there are no non-zero characteristic $p$ Brauer relations supported only at $p$-hypo-elementary subgroups.
	\end{enumerate}
\end{thm}
\begin{proof}
	The first statement is a consequence of Conlon's induction theorem as we show later in Theorem \ref{basisphypo}. Given {\it i)}, both {\it ii)} and {\it iii)} are automatic.
\end{proof}
The first part also holds with $\z_p$ replaced by $\F_p$ or $\zp$.

Note that there are no non-zero characteristic $p$ Brauer relations for $p$-hypo-elementary groups, and this is only true of such groups. As before, the theorem gives rise to privileged relations in characteristic $p$:

\begin{mydef}For any group $G$ and prime $p$, write $\one_{\z_p, G}=\sum_{H \in \hyp_p(G)}\alpha_H \one_{\z_p,H}\up^G_H$ with $\alpha_H\in \q$ uniquely by Theorem \ref{dimbrelp}. Since $\BR_{\z_p}(G)=\BR_p(G)$ (Lemma \ref{BRp=Brzpapp}), 
	\[\theta_{\text{Con,$G$}} = [G]-\hspace{-5pt}\sum_{H \in \hyp_p(G)} \hspace{-5pt}\alpha_H[H] \]
	is a rational Brauer relation in characteristic $p$, which we refer to as the \emph{Conlon relation} of $G$. Note, $\theta_{\text{Con,$G$}}$ is identically zero if and only if $G$ is $p$-hypo-elementary. The Conlon relation is the unique $p$-relation supported only at $G$ and $p$-hypo-elementary subgroups. However, the Conlon relation need not be unique amongst characteristic zero relations supported at these groups.
\end{mydef}
As before, when it is clear that we are referring to $G$-relations, for a subgroup $H\le G$, we denote $\theta_{\text{Con,$H$}}\up^G$ simply by $\theta_{\text{Con,$H$}}$. All characteristic $p$ relations are rational linear combinations of Conlon relations:
\begin{lemma}\label{quotbasis}
	Let $G$ be a finite group and $p$ a prime. Then,
	\begin{enumerate}[i)]
		\item a basis of $\BR_p(G)$ is formed by the set $\{ \theta_{\text{Con,H}}\}$ as $H$ runs over conjugacy classes of non-$p$-hypo-elementary groups,
		\item this can be extended to a basis of $\BR_0(G)$ by adding the Artin relations $\theta_H$ as $H$ runs over conjugacy classes of non-cyclic $p$-hypo-elementary groups.
	\end{enumerate}
\end{lemma}
\begin{proof}
	The proof of {\it i)} is as in Lemma \ref{artinbasisofbrauerrel}. For {\it ii)}, in addition use Corollary \ref{dimbrel0}, Theorem \ref{dimbrelp}.
\end{proof}
\begin{example}\label{tsD2p} Let $G=D_{2p}=C_p\rti C_2$ be the dihedral group of order $2p$ for $p$ an odd prime. If $\ell$ is any prime, then
	\begin{equation} \{H\le_GG \ |\ H\text{ is $\ell$-hypo-elementary}\}=\begin{cases}\{1\},C_2,C_p & \ell \neq p\\ \{1\},C_2,C_p, D_{2p} & \ell = p
	\end{cases},\label{phypolist}  \end{equation}
	and so $\dim_\q A(\z_{\ell}[G],\perm)$ is $3$ unless $\ell=p$ when it is $4$. A basis $S$ of $A(\z_p[G],\perm)$ is formed by
	\[ S=\begin{cases}\one_{\{1\}}\up^G,\one_{C_2}\up^G,\one_{C_p}\up^G & \ell \neq p\\ \one_{\{1\}}\up^G,\one_{C_2}\up^G,\one_{C_p}\up^G, \one_G & \ell = p
	\end{cases}.\]	
\label{S3p} Since $G$ has up to conjugacy four subgroups, of which three are cyclic, $\rk \BR_0(G)=1$. Let $\theta \in \br_0(G)$ be the relation
	\[\theta \colon 2[G]-[C_p]-2[C_2]+[1].\]
	Then $\theta=2\theta_G$. As $\theta$ is indivisible as an element of $b(G)$, we find $\br_0(G)=\theta \cdot\z$.
	
Since $\rk \br_\ell(G)$ is the number of conjugacy classes of non-$p$-hypo-elementary subgroups, by \eqref{phypolist}, $\rk\br_\ell(G)$ is also one unless $\ell= p$ when it is zero. Given that $\br_\ell(G)\subseteq \br_0(G)$ (Lemma \ref{BRp=Brzpapp}), we find
	\[ \br_\ell(G)=\begin{cases} \theta \cdot \z & \text{if } \ell \neq p\\
	0 & \text{if } \ell=p\end{cases}.\]
	The Conlon relation $\theta_{\text{Con}, G}$ is equal to the Artin relation unless $\ell=p$ when it is zero.
\end{example}

\subsection{Regulator constants}\label{regcon}
In this section, we recall how to associate to a characteristic zero Brauer relation, a function on (nice) $\R[G]$-lattices called its regulator constant.

We follow the construction given in \cite{DDRegConstParity} for an arbitrary PID $\mathcal{R}$ of characteristic zero. We are only ever concerned with $\mathcal{R}=\z,\zp$ or $\z_p$. Let $\mathcal{K}$ denote the field of fractions of $\mathcal{R}$. 
\begin{mydef}\label{selfdual}
An $\R[G]$-lattice $M$ is called \emph{rationally self-dual} if $M\ot \mathcal{K}$ is self-dual, i.e.\ $M\ot \mathcal{K}$ is isomorphic to its linear dual $\Hom_\mathcal{K}(M\ot \mathcal{K}, \mathcal{K})$ as $\mathcal{K}[G]$-modules. This is equivalent to the existence of a non-degenerate $G$-invariant inner product on $M\ot \mathcal{K}$. If an inner product on $M\ot \mathcal{K}$ exists, there is a restricted $G$-invariant inner product on $M$.

Rational self-duality is preserved under induction, restriction, inflation and deflation, as well as taking tensor products of two rationally self-dual lattices. We say that an element of $A(\R[G])$ is \emph{rationally self-dual} if it can be written as a linear combination of self-dual lattices. We denote the subring of self-dual lattices by $a(\R[G],\sd)$ and define $A(\R[G],\sd)$ accordingly.
\end{mydef}
A rationally self-dual lattice $M$ need not be linearly self-dual, i.e.\ a rationally self-dual $M$ need not be isomorphic to $\Hom(M,\R)$. If $\R=\z$ or $\zp$, then, as all $\q[G]$-modules are self-dual, all $\R[G]$-lattices are rationally self-dual.
\begin{mydef}\label{regdef}
Let $G$ be a finite group and $\theta= \sum_{i} [H_i]-\sum_{j} [H'_j]\in \br_0(G)$ be an integral characteristic zero Brauer relation of $G$. Given a rationally self-dual $\R[G]$-lattice $M$, fix a choice of non-degenerate $G$-invariant inner product $\la \ \, , \ \ra$ on $M$. The \emph{regulator constant} of $\theta$ evaluated at $M$ is then
\[ C_{\theta}(M)=\frac{\prod_i\det\left( \frac{1}{|H_i|}\la \ \, , \ \ra |_{M^{H_i}} \right)}{\prod_j\det\left( \frac{1}{|H'_j|}\la \ \, , \ \ra |_{M^{H'_j}}\right)} \in \mathcal{K}\tii/(\mathcal{R}\tii)^2. \]
This is independent of the choice of $\la \ \, , \ \ra$ as an element of $\mathcal{K}\tii/(\mathcal{R}\tii)^2$ (see \cite[Thm.\ 2.17]{DDRegConstParity}). For $M$ a $\z[G]$ or $\zp[G]$-lattice, we may take the pairing to be positive definite and so for all characteristic zero Brauer relations $\theta$ and modules $M$ we have $C_\theta(M)>0$.\label{no-1}
\end{mydef}
\begin{remark}\label{lookattsmodulessuggestion}
When evaluating regulator constants at the trivial module, the formula simplifies. For example, if $\theta=\sum_i [H_i]-\sum_j [H'_j]$, then 
\begin{equation}C_\theta(\one _G)= \frac{\prod_i \frac{1}{|H_i|}}{\prod_j \frac{1}{|H'_j|}}=\frac{\prod_j|H'_j|}{\prod_i{|H_i|}}.  \label{oneform}
\end{equation}
This formula can be extended to permutation modules due to the formalism of regulator constants provided by the next lemma. That regulator constants of permutation modules can be made explicit in this way is crucial in the proof of Theorem \ref{tspairingnondeg}.
\end{remark}
\pagebreak[3]
\begin{lemma}[{\rm \cite{DDRegConstParity}}]\label{regcontoolkit} Let $G$ be any finite group and $H$ a subgroup and let $\mathcal{R},\mathcal{K}$ be as above. Assume throughout that all modules are rationally self-dual. Then,
\begin{enumerate}[i)] \label{regconstmult}\label{inductregconst}\label{inductmod}
	\item \label{multinmod} if $M,N$ are two $\R[G]$-lattices, then for any Brauer relation $\theta$ of $G$, $C_\theta(M\oplus N)=C_\theta(M)C_\theta(N)$,
	\item if $\theta,\theta'$ are two Brauer relations for $G$ and $M$ any $\mathcal{R}[G]$-lattice, then \linebreak $C_{(\theta+\theta')}(M)=C_{\theta}(M)C_{\theta'}(M)$, 
	\item if $M$ is a $\mathcal{R}[G]$-lattice, then $C_{\theta\up^G_H}(M)=C_{\theta}(M\down^G_H)$, 
	\item if $M$ is a $\mathcal{R}[H]$-lattice, then  $C_{\theta}(M\up^G_H)=C_{\theta\down_H^G}(M)$,
	\item if $H$ is normal, then given a relation of $G/H$ and a $\R[G]$-lattice $M$, \linebreak$C_{\inf ^G_{G/H}\theta}(M)=C_{\theta}(\defl^G_{G/H}M)$,
	\item for any inclusion $\mathcal{R} \hookrightarrow \mathcal{T}$, with $\mathcal{T}$ a PID, any relation $\theta$, and $\mathcal{R}[G]$-lattice $M$, we have \linebreak $C_\theta(M)=C_\theta(M\ot \mathcal{T})\in \mathcal{L}\tii/(\mathcal{T}\tii)^2$ where $\mathcal{L}$ denotes the field of fractions of $\mathcal{T}$. \label{regextscalars}
\end{enumerate}
\end{lemma}

\begin{notat}
	By definition, regulator constants of rationally self-dual $\z_p[G]$-modules take values in $\q_p^\ti/(\z_p^\ti)^2$. Let $v_p \colon \q_p^\ti \to \z$ denote the usual $p$-adic valuation. This descends to a function $v_p\colon \q_p^\ti/(\z_p^\ti)^2\to \z$. For any prime $\ell \neq p$, we also have a ``valuation at $\ell$'' function $\q_p^\ti/(\z_p^\ti)^2\to \z/2\z$,  which we also denote by $v_\ell$.
\end{notat}
\begin{remark}\label{productp}
 Since the regulator constant of a $\z[G]$-lattice is always a positive rational number (see Definition \ref{no-1}), Lemma \ref{regcontoolkit} {\it vi)} shows that the regulator constant $C_\theta(M)$ of a $\z[G]$-lattice $M$ is a function of the values $v_p(C_\theta(M\ot \z_p))$ as $p$ runs over all primes.
\end{remark} 
The following observation will be crucial:
\begin{lemma}[{\rm \cite[Lem.\ 3.6]{BartelDih}}]\label{prelnopterms}
If $G$ is a finite group and $\theta$ a relation in characteristic $\ell$, then for any prime $p$ (possibly equal to $\ell$) and $M$ any $\z[G]$ or rationally self-dual $\z_{p}[G]$-lattice we have
\[ v_\ell(C_\theta(M))=0.\]
\end{lemma}
\begin{remark}\label{lphypo}
	If $G$ is a finite group and $p$ is a prime not dividing the order of $G$, then the $p$-hypo-elementary subgroups of $G$ are the cyclic subgroups and so all characteristic zero relations are characteristic $p$ relations (Lemma \ref{quotbasis}). Thus Lemma \ref{prelnopterms} shows that the only prime powers appearing in regulator constants divide the order of the group. If $G$ itself is $\ell$-hypo-elementary, then, for $p\neq \ell$, all its $p$-hypo-elementary subgroups are cyclic and so all its characteristic $0$ relations are relations in characteristic $p$ and its regulator constants are always $\ell$\th powers.
\end{remark}
\section{Pairings from regulator constants}\label{paireg}

In this section, we remark that the construction of regulator constants canonically defines a pairing between Brauer relations and rationally self-dual $\z_p[G]$-lattices. This pairing has obvious Brauer relations and lattices which must lie in the kernel, but it is unclear what the kernels should be in general. Sections \ref{regconstinv} and \ref{nonvanartin} can be seen as partial results in this direction. Finally, we show that non-degeneracy of such pairings leads to methods to determine the isomorphism classes of permutation modules.
\subsection{The regulator constant pairing}\label{regpai}
\begin{constr}\label{regpairingconstr}
Let $G$ be any finite group and $p$ a prime. The map
	\begin{align*}
v_p(C_{(-)}(-)) \colon \br_0(G) \ti a(\z_p[G],\sd)&\longrightarrow \z\\
(\theta, M) \qquad \qquad& \longmapsto v_p(C_\theta(M)),
\end{align*}
is bi-additive (Lemma \ref{regcontoolkit} {\it i),ii)}). %
Extending $\q$-linearly we get a pairing,
\[  v_p(C_{(-)}(-)) \colon \BR_0(G) \ti A(\z_p[G],\sd)\longrightarrow \q ,\]
 which we also denote by $v_p(C_{(-)}(-))$ and which we call the \emph{regulator constant pairing}. By Lemma \ref{prelnopterms}, this factors as
\[ v_p(C_{(-)}(-)) \colon \BR_0(G)/\BR_p(G) \ti A(\z_p[G],\sd)  \longrightarrow \q. \]
\end{constr}
In Section \ref{D2pfullex}, we calculate the full regulator constant pairing for dihedral groups $D_{2p}$ with $p$ odd, one of the few families of groups where a classification of all indecomposable lattices exists.
\begin{remark}\label{indfromcycsgps}The pairing $v_p(C_{(-)}(-))$ is far from non-degenerate; $A(\z_p[G],\sd)$ is often infinite dimensional whilst $\BR_0(G)$ is always finite dimensional. Explicit elements of the right kernel are given by taking a lattice $M\up^G_H$ induced from a cyclic subgroup $H$. This pairs to zero with all relations since
	\[ v_p(C_\theta(M \up^G_H))=v_p(C_{\theta\down _H^G}(M))=0, \]
	where first equality is Lemma \ref{regconstmult} {\it iv)} and the second is because cyclic groups have no non-zero Brauer relations (Corollary \ref{dimbrel0}). The behaviour of the left kernel is less clear:
\end{remark}
\begin{question}\label{regpairingquest}
	Are there groups for which the left kernel of $v_p(C_{(-)}(-))\colon \BR_0(G)/\BR_p(G) \ti A(\z_p[G],\sd) \to \q$ is non-trivial?
\end{question}

It is also interesting to replace $\z_p$ with other rings. For $\zp$ and so also for $\z_p$, if $G$ has cyclic Sylow $p$ subgroups, then the left kernel is trivial (see Theorem \ref{tspairingnondeg}). Outside of this case, things are less clear and there are some very small groups (e.g.\ $C_3\ti C_3 \ti S_3$ when $p=3$) for which we have been unable to determine the left kernel. There are groups with Brauer relations which pair trivially with all summands of permutation modules but are not relations in characteristic $p$ (namely for $C_3\ti C_3 \ti S_3$, see Section \ref{counterexamples}). On the other hand, the kernel is never all of $\BR_0(G)$ (Theorem \ref{nonvanishing}).
\begin{lemma}\label{sillylemma}
	Let $\theta$ be a relation of a finite group $G$. For any prime $p$ the following are equivalent,
	\begin{enumerate}[i)]
		\item $v_p(C_\theta(-))\colon A(\z_p[G],\sd) \to \q$ vanishes identically,
		\item $v_p(C_{\theta\down _H}(-))\colon A(\z_p[H],\sd)\to \q$ vanishes identically for all conjugacy classes of $p$-hypo-elementary subgroups $H\le G$.
	\end{enumerate} 
\end{lemma}
\begin{proof}
		For the forward direction, use that $v_p(C_{\theta\down_H}(M))=v_p(C_\theta(M\up^G))=0$ (Lemma \ref{regcontoolkit} {\it iv)}). For the reverse, write $\one= \sum_{H\in \hyp_p(G)}\alpha_H \one_H \up^G$ as in the Conlon relation. Then for any rationally self-dual $\z_p[G]$-lattice 
		\[M =M\ot \one = \sum_{H\in \hyp_p(G)} \alpha_H \cdot( M\otimes \one \up^G_{H})=\sum_{H\in \hyp_p(G)} \alpha_H\cdot M\down_H\up^G_{H},\] where the $M\down_H$ are rationally self-dual. Then, 
		\[v_p(C_{\theta}(M))=\sum_{H\in \hyp_p(G)} \alpha_H v_p(C_{\theta}(M\down_H\up^G_{H}))=\sum_{H\in \hyp_p(G)} \alpha_H v_p(C_{\theta\down_{H}}(M\down_H))=0\].	
\end{proof}
\begin{lemma}\label{reductiontophypo}
	For any finite group $G$, the following are equivalent:
	\begin{enumerate}[i)]
		\item the left kernel of $v_p(C_{(-)}(-)) \colon \BR_0(G)/\BR_p(G) \ti A(\z_p[G],\sd) \to \q$ is trivial,
		\item the left kernel of $v_p(C_{(-)}(-)) \colon \BR_0(G/N)/\BR_p(G/N) \ti A(\z_p[G/N],\sd) \to \q$ is trivial for all normal subgroups $N$.
	\end{enumerate}
\noindent Moreover, both are implied by
\begin{enumerate}
	\item[iii)] the left kernels of $v_p(C_{(-)}(-))\colon  \BR_0(H) \ti A(\z_p[H],\sd)\to \q$	are trivial for all isomorphism classes of $p$-hypo-elementary subgroups $H\le G$.
\end{enumerate}
\end{lemma}
\begin{proof}
	To see {\it i)} $\implies$ {\it ii)}, suppose that the left kernel of the pairing for $G$ is trivial and let $\theta$ be a relation for $G/N$ which is not a relation in characteristic $p$. Since $\text{defl}\circ \text{inf}=\id$ and both take characteristic $p$ relations to characteristic $p$ relations, $\inf \theta$ must also not be a $p$-relation. So, by assumption, there exists an $M$ for which $0 \neq v_p(C_{\inf \theta}(M))$. By Lemma \ref{regcontoolkit}\,{\it v)}, $v_p(C_{\inf \theta}(M))=v_p(C_\theta(\defl M))\neq 0$ and $v_p(C_{\theta}(-))$ doesn't vanish identically. The reverse direction is automatic.
	
	Now assume {\it iii)}. By Lemma \ref{sillylemma}, $v_p(C_\theta(-))$ vanishes if and only if $v_p(C_{\theta\down_H}(-))$ vanishes for all $H$. But if $\theta$ is not a relation in characteristic $p$, there exists a $p$-hypo-elementary subgroup $H$ for which $\theta\down_H \neq 0$ (use Lemma \ref{quotbasis} {\it i)}) and so $ v_p(C_{\theta\down_H}(-))$ doesn't vanish.
\end{proof}
\subsection{The permutation pairing}\label{perpai}
In this section, we study the restriction of the regulator constant pairing to permutation modules. Here we have a chance to be much more explicit as Theorem \ref{dimbrelp} describes a basis of $A(\z_p[G],\perm)$, and regulator constants of permutation modules are easy to calculate.
 \begin{notat}
 	Let $P(G)$ denote the free $\q$-vector space on the set of conjugacy classes of non-cyclic $p$-hypo-elementary subgroups.
 \end{notat}
\begin{remark}\label{kernels}
	As in Remark \ref{indfromcycsgps}, if we restrict to $A(\z_p[G],\perm)$, then $v_p(C_{(-)}(-))$ factors as
	\[  v_p(C_{(-)}(-)) \colon \BR_0(G)/\BR_p(G) \ti A(\z_p[G],\perm)/A(\z_p[G],\cyc) \to \q. \]
	Lemma \ref{quotbasis}\,{\it ii)} demonstrates that $P(G)$ is canonically isomorphic to $\BR_0(G)/\BR_p(G)$ (via \linebreak[4] $H \mapsto \theta_H$). On the other hand, Theorem \ref{ait} shows that $P(G)$ is canonically identified with \linebreak[4] $A(\z_p[G],\perm)/A(\z_p[G],\cyc)$ by sending $H$ to $\one \up^G_H$.
	
	It is not true that the spaces can be identified before factoring $v_p$. Indeed, $\BR_0(G)$ is of dimension equal to the number of non-cyclic subgroups, whereas $A(\z_p[G],\perm)$ is of dimension equal to the number of non-$p$-hypo-elementary subgroups.
\end{remark}
\begin{constr}\label{permpairingconstr}
Via these canonical identifications, we may consider the restricted pairing of Remark \ref{kernels} as a pairing
\[ \la \ \, , \ \ra_\perm \colon P(G)\ti P(G) \longrightarrow \q , \]
sending $(H,K)$ to $v_p(C_{\theta_H}(\one\up^G_K))$. We call $\la\ \, , \ \ra_\perm$ the \emph{permutation pairing}.
\end{constr}
\begin{lemma}\label{sym} For any finite group $G$ and prime $p$, $\la \ \, , \ \ra_\perm \colon P(G)\ti P(G) \longrightarrow \q$ is symmetric.
\end{lemma}
\begin{proof}\vspace{-\topsep}\vspace{3pt}
For any two subgroups $H$ and $K$ of $G$, Lemmas \ref{mackeyforartrel}, \ref{regconstmult} show that
	\begin{align*}
	C_{\theta_H\up^G}(\one_K\up^G)&=C_{\theta_H\up^G\down_K}(\one_K)\\
	&= \prod_{g \in H \ba G /K}\hspace{-5pt}C_{\theta_{H^g\cap K}}(\one_{H^g\cap K}).
	\end{align*}
	Whilst,
	\begin{align*}
	C_{\theta_K\up^G}(\one_H\up^G)&= C_{\theta_K}(\one_H\up^G\down_K)\\
	&= C_{\theta_K}\left(\sum_{g \in H \ba G/K}\hspace{-5pt}\one_{H^g \cap K }\up^K\right)\\
	&=\prod_{g \in H \ba G/K}\hspace{-5pt}\left( C_{\theta_{H^g\cap K}}(\one_{H^g\cap K})\right).
	\end{align*}
\end{proof}
\begin{remark}\label{permpairformula}
	Along the same lines, Lemma \ref{regcontoolkit} and \eqref{mackeyforartreleq} show that, for $H,K\le G$,
	\begin{align}
\la H,K\ra_\perm &\hspace{-2pt}:= v_p(C_{\theta_{H}\up^G}(\one _{K}\up^G))\nonumber\\
&= v_p(C_{\theta_{H}\up^G\down_{K}}(\one _{K}))\nonumber \\
&=\sum_{g \in H\bac G/K}\hspace{-5pt} C_{\theta_{H^g} \down_{H^g\cap K}\up^{K}}(\one _{ K})\nonumber\\
&=\sum_{g \in H\bac G/K}\hspace{-5pt} C_{\theta_{H^g\cap K}\up^{K}}(\one _{K})\nonumber\\
&=\sum_{g \in H\bac G/K}\hspace{-5pt} v_p(C_{\theta_{H^g\cap K}}(\one _{H^g\cap K})). \label{doublecosetformula}
\end{align}
Combining this with \eqref{oneform} of p\pageref{oneform} gives a formula for the permutation pairing.
\end{remark}

It is tempting to ask if permutation pairing is non-degenerate for all groups $G$. This proves too naive, in Section \ref{counterexamples}, we exhibit a family of groups for which the permutation pairing is degenerate (e.g.\ $C_3\ti C_3\ti S_3$ when $p=3$). Analogously to Lemmas \ref{sillylemma}, \ref{reductiontophypo} we have:
\begin{lemma}
	Let $\theta$ be a relation of a finite group $G$. For any prime $p$ the following are equivalent,
\begin{enumerate}[i)]
	\item $v_p(C_\theta(-))\colon A(\z_p[G],\perm) \to \q$ vanishes identically,
	\item $v_p(C_{\theta\down _H}(-))\colon A(\z_p[H],\perm)\to \q$ vanishes identically for all conjugacy classes of $p$-hypo-elementary subgroups $H\le G$.
\end{enumerate}
\end{lemma}
\begin{lemma}\label{reductiontophypoperm}\label{permredtophypo}
	For any finite group $G$, the following are equivalent:
	\begin{enumerate}[i)]
		\item the permutation pairing is non-degenerate,
		\item the permutation pairing of $G/N$ is non-degenerate for all $N \unlhd G$.
\item[] \hspace{-0.6cm} Moreover, both are implied by
		\item[iii)] the permutation pairing of $H$ is non-degenerate for all $p$-hypo-elementary subgroups $H$.
	\end{enumerate}
\end{lemma}
The proofs are identical to before. As a result, we see that infinitely many groups exist where the permutation pairing is degenerate, for example, when $p=3$, all groups with a $C_3\ti C_3\ti S_3$ quotient.
We prove two main results on the permutation pairing. Theorem \ref{tspairingnondeg} states that the permutation pairing is non-degenerate for all groups with cyclic Sylow $p$-subgroups. Whilst, for arbitrary $G$, Theorem \ref{nonvanishing} states that the permutation pairing is not the zero pairing (unless $P(G)=0$). 

This leaves many open questions. For example:
\begin{question}\label{classificationquest}
	Can one describe the groups for which the permutation pairing is degenerate?
\end{question}
It would also be interesting to know of the existence of a group with degenerate permutation pairing but for which the regulator constant pairing has trivial left kernel. There are also many hard problems which arise from considering these pairings integrally.
\subsection{Regulator constants as invariants of permutation modules}\label{reginv}
The non-degeneracy of the permutation pairing is a measure of the strength of regulator constants as invariants of permutation modules. In this section, we show that the isomorphism class of an arbitrary $\z_p[G]$-permutation module is determined by the isomorphism class of its extension of scalars to $\q_p$ and regulator constants if and only if the permutation pairing is non-degenerate. 
\begin{constr}
Let $P''(G)$ denote the free $\q$-vector space on conjugacy classes of cyclic subgroups. Artin's induction theorem (Thm.\ \ref{ait}) states that there is a canonical isomorphism $P''(G) \overset{\sim}{\to} A(\q[G])$ sending $H \to \one_{\q,H} \up^G$. In the same way, $A(\z_p[G],\cyc)$ is also canonically identified with $P''(G)$. Define a pairing
	\begin{align*} \la \ \, , \ \ra_\char \colon P''(G)\ti P''(G)&\longrightarrow \q \\
	( H,K) & \longmapsto \la \one_{\z_p,H}\up^G\ot \q_p , \one_{\z_p,K} \up ^G \ot \q_p \ra _\char , \end{align*}
	where the final inner product is the usual pairing given by character theory. Then, $\la \ \, , \ \ra_\char$ is symmetric and is non-degenerate by Artin's induction theorem.
\end{constr}
\begin{constr}	Let $P'(G)$ denote the free $\q$-vector space on conjugacy classes of $p$-hypo-elementary subgroups of $G$. We define the pairing
	\begin{align*} \la \ \, , \ \ra_* \colon P'(G)\ti P'(G) &\longrightarrow \q \\
	(H,K) &\longmapsto \begin{cases}
\la \one_{\z_p,H}\up^G\ot \q_p , \one_{\z_p,K} \up ^G \ot \q_p \ra _\char &   \text{if $H$ is cyclic}\\
	v_p(C_{\theta_H}(\one_{\z_p,K}\up ^G)) & \text{if $H$ is non-cyclic}
	\end{cases} . \end{align*}
 This extends both $\la \ \, , \ \ra_\perm$ and $\la \ \, , \ \ra_\char$.
\end{constr}
\begin{remark}\label{extension}
	The pairing $\la \ \, , \ \ra_*$ is chosen so that, via the identification $P'(G) \cong A(\z_p[G],\perm)$, in the second variable, the construction extends to a pairing $P'(G)\ti A(\z_p[G])\to \q$ on the full representation ring (cf.\ Remark \ref{regvsspecies}).
\end{remark}
\begin{lemma}\label{tsdet}For any finite group $G$, the following are equivalent,
	\begin{enumerate}[i)]
		\item the permutation pairing of $G$ is non-degenerate,
		\item the pairing $\la \ \, , \ \ra_*$ is non-degenerate,
		\item the isomorphism class of an arbitrary permutation module over $\z_p$ is determined by
		\begin{enumerate}[a)]
			\item the isomorphism class of $M\ot \q_p$, and
			\item the valuations of the regulator constants $v_p(C_{\theta_H}(M))$ as $H$ runs over elements of $\nchyp_p(H)$.
		\end{enumerate}
	\end{enumerate}
\end{lemma}
\begin{proof}\vspace{-\topsep}
For equivalence of {\it i)} and {\it ii)}, note that, for any cyclic subgroup $K$, $v_p(C_{\theta_H}(\one \up^G_K))=0$ (Remark \ref{indfromcycsgps}). Thus, if we order the canonical basis of $P'(G)$ so that the cyclic subgroups come before the non-cyclic $p$-hypo-elementary subgroups, then the matrix representing $\la \ \, , \ \ra_*$ is block upper triangular, with diagonal blocks given by the matrices representing $\la \ \, , \ \ra_\char$ and the permutation pairing respectively. The former is always invertible so $\la \ \, , \ \ra_*$ is non-degenerate if and only if the permutation pairing is.

The equivalence of {\it ii)} and {\it iii)} is automatic.
\end{proof}
\begin{example}
	Let $G=D_{2p}$. Up to conjugacy, the $p$-hypo-elementary subgroups of $G$ are $S=\{\{1\}, C_2,C_p,D_{2p}\}$. Applying \eqref{oneform} of p\pageref{oneform} to $\theta_G$ as calculated in Example \ref{artinrelcyclic}, we find $v_p(C_{\theta_G}(\one_G))=-1/2$. Thus, the matrix representing $\la \ \, , \ \ra_*$ with respect to the basis of $P'(G)$ given by $S$ is:
	\begin{figure}[H]
	\begin{tikzpicture}[
	every left delimiter/.style={xshift=.5em},
	every right delimiter/.style={xshift=-.5em},
	]
	
	\matrix [ matrix of math nodes,left delimiter={( },right delimiter={ )},row sep=0.1cm,column sep=0.1cm] (U) { 
		2p &p &2& 1\\
		p & (p+1)/2&1&1\\
		2 & 1&2&1\\
		0 &0 &0 &-1/2\\
	};

	\draw (1,1.3) -- (1,-1.2);
	\draw (-2,-0.6) -- (2,-0.6);
	
	\node at (-1.7,1.7) {$\{1\}$};
	\node at (-0.5,1.7) {$C_2$};
	\node at (0.7,1.7) {$C_p$};
	\node at (1.55,1.7) {$D_{2p}$};
	\node[left=10pt  of U-1-1]  {$\{1\}$}; 
	\node[left=14pt  of U-2-1]  {$C_2$}; 
	\node[left=15pt  of U-3-1]  {$C_p$};
	\node[left=14pt  of U-4-1]  {$D_{2p}$};
	\node at (2.4,-1) {.};
	\end{tikzpicture}
\end{figure}\noindent
	In Section \ref{D2pfullex}, we extend this to allow arbitrary $\z_p[D_{2p}]$-lattices.
\end{example}

\section{Non-degeneracy of the permutation pairing for groups with cyclic Sylow $p$-subgroups}\label{regconstinv}
 In this section we prove:
\begin{thm}\label{tspairingnondeg}
	Let $G$ be a finite group and $p$ a prime such that $G$ has cyclic Sylow $p$-subgroups. Then, the permutation pairing
	\[ v_p(C_{(-)}(-)) \colon \BR_0(G)/\BR_p(G) \ti A(\z_p[G],\perm)/A(\z_p[G],\cyc) \to \q   \]
	is non-degenerate.
\end{thm}
As a result, for such groups $G$, the regulator constant pairing has trivial left kernel and we find:
\begin{corol}\label{cyctsdet}
	Let $G$ be a finite group and $p$ a prime for which the Sylow $p$-subgroups of $G$ are cyclic. Then, the isomorphism class of an arbitrary $\z_p[G]$-permutation module $M$ is determined by,
	\begin{enumerate}[i)]
		\item the isomorphism class of $M\ot \q_p$,
		\item the valuations of the regulator constants $v_p(C_{\theta_H}(M))$ as $H$ runs over elements of $\nchyp_p(H)$.
	\end{enumerate}
\end{corol}
\begin{proof}
This follows by Lemma \ref{tsdet}.
\end{proof}
The proof of Theorem \ref{tspairingnondeg} reduces to the case of $G$ $p$-hypo-elementary. Since all $p$-hypo-elementary groups with cyclic Sylow $p$-subgroups are of the form $C_{p^k}\rti C_n$ with $(n,p)=1$ we can then perform an explicit calculation.
\subsection{GCD matrices}
We first state and prove a purely combinatorial statement. Since this may be of limited independent interest this subsection is self contained.

\begin{notat}\label{matrixnotat}
	For a natural number $n$ and divisor $s$ of $n$, denote by
	\begin{itemize}
		\item $D'(n)$ the set of divisors of $n$ (ordered increasingly),
		\item $D(n,s)\subset D'(n)$ the set of divisors of $n$ not dividing $s$,
		\item $N(n)$ the symmetric matrix with rows and columns indexed by elements of $D'(n)$ and \linebreak[4] $(d_1,d_2)$\th entry given by $\gcd(d_1,d_2)$,
		\item $M(n,s)$ the symmetric matrix with rows and columns indexed by elements of $D(n,s)$ and $(d_1,d_2)$\th entry given by $(\gcd(d_1,d_2)-\gcd(d_1,d_2,s))$.
	\end{itemize}
\end{notat}
\begin{example}\label{12,2}
	If $n=12$ and $s=2$ then $D(12,2)=\{ 3,4,6,12  \}$ and 
	\[M(12,2) = \begin{blockarray}{ccccc}
	&3&4& 6&12\\
	\begin{block}{c(cccc@{\hspace*{5pt}})}
	3 & \hspace{2pt}2 &0&2 &2\hspace{2pt}\\		
	4&\hspace{2pt}0& 2&0&2\hspace{2pt}\\
	6 &\hspace{2pt}2&0 & 4& 4\hspace{2pt}\\
	12&\hspace{2pt}2 & 2 &4 &10\hspace{2pt}\\
	\end{block}
	\end{blockarray}\ ,\]
	which has full rank.
\end{example}
\begin{remark}
	Matrices of the form $N(n)$ are called GCD matrices and are always invertible (not necessarily integrally, see Lemma \ref{gcdmatrixeasy}). Although matrices defined in a similar way to $M(n,s)$ have been studied (see \cite{GCD1,GCD2}),  we have been unable to find results in the literature that directly cover matrices of the form $M(n,s)$. For this reason, we have included a full calculation of their determinants and thus invertibility. First we recall the proof of the determinant formula for $N(n)$.
\end{remark}
\begin{lemma}\label{gcdmatrixeasy}
	For any natural number $n$, the matrix $N(n)$ has determinant $\prod_{d \in D'(n)} \varphi(d)$, where $\varphi$ denotes Euler's totient function, and thus is always of full rank.
\end{lemma}
\begin{proof}
	For $n=p^e$ a prime power,
	\[N(p^e) = \left(\begin{array}{c c c c}
	1 & 1 & ...&1 \\
	1 & p & ... & p \\
	\vdots & \vdots & & \vdots\\
	1 & p & ... & p^e
	\end{array}\right).\]
 If $e>1$, expanding the final column shows that $\det(N(p^e))=(p^e-p^{e-1})\det(N(p^{e-1}))=\varphi(p^e)\det(N(p^{e-1})),$ as any $(e-1)\ti(e-1)$ minor containing the first $(e-1)$ terms of the last two rows is not of
 full rank. Inductively this shows the determinant formula for prime powers.
	
	Now let $s=rt$ with $(r,t)=1$. Then, using the bijection $D'(rt)\leftrightarrow D'(r)\ti D'(t)$, after simultaneous permutation of rows and columns (which preserves the determinant) $N(s)$ is of the form:
		\begin{figure}[H]
		\begin{tikzpicture}[
		every left delimiter/.style={xshift=.5em},
		every right delimiter/.style={xshift=-.5em},
		]
		\node at (-6.2,0) {$N(s)=$};
		\matrix [ matrix of math nodes,left delimiter={( },right delimiter={ )},row sep=0.1cm,column sep=0.1cm] (U) { 
 \gcd(u_1,u_1)N(r) &\gcd(u_1,u_2)N(r)& ... &\gcd(u_1,u_k)N(r)\\
\gcd(u_2,u_1)N(r)& \gcd(u_2,u_2)N(r)&...&\gcd(u_2,u_k)N(r)\\
\vdots&\vdots & \ddots& \vdots \vspace{1cm}\\
\gcd(u_k,u_1)N(r) & \gcd(u_k,u_2)N(r) &... &\gcd(u_k,u_k)N(r)\\
		};

				\node at (6,0) {$=N(r)\ot N(t)$,};

		\draw (U-1-1.south west) -- (U-1-4.south east);
		\draw (U-2-1.south west) -- (U-2-4.south east);
		\draw (U-4-1.north west) -- (U-4-4.north east);
		
		\draw (U-1-1.north east) -- (U-4-1.south east);
		\draw (U-1-2.north east) -- (U-4-2.south east);
		\draw (U-1-4.north west) -- (U-4-4.south west);
		
		\node[above=0pt  of U-1-1]  {$u_1$}; 	
		\node[above=0pt  of U-1-2]  {$u_2$}; 
		\node[above=12pt  of U-1-3]  {$\dots$}; 
		\node[above=0pt  of U-1-4]  {$u_k$}; 
		
		\node[left=10pt  of U-1-1]  {$u_1$}; 
		\node[left=10pt  of U-2-1]  {$u_2$}; 
		\node[left=47pt  of U-3-1]  {$\vdots$};
		\node[left=10pt  of U-4-1]  {$u_k$};
		
		\end{tikzpicture}
	\end{figure}\noindent
	writing $u_i$ for the elements of $D'(t)$. If $A,B$ are matrices of dimension $m,n$ respectively, then their tensor product satisfies the familiar formula
	\[\det(A\ot B)=\det(A)^n\det(B)^m.\]
	Applying this inductively, using the bijection $D'(rt)\leftrightarrow D'(r)\ti D'(t)$,	
	\begin{align*}
	\det(N(r)\ot N(t))&=\left(\prod_{d\in D'(r)}\hspace{-3pt}\varphi(d)\right)^{|D'(t)|}\cdot\left(\prod_{l\in D'(t)}\hspace{-3pt}\varphi(l)\right)^{|D'(r)|}\\
	&= \prod_{d \in D'(r)} \left( \varphi(d)^{|D'(t)|}\prod_{l\in D'(t)}\hspace{-3pt} \varphi (l)\right)\\
	&=\prod_{d \in D'(r)}\prod_{l \in D'(t)}\hspace{-3pt}\varphi(d)\varphi(l)\\
	&=\prod_{w \in D'(rt)}\hspace{-5pt}\varphi(w),
	\end{align*}
	as required.
\end{proof}

\begin{lemma}
	\label{gcdmatinvertible}
	The matrix $M(n,s)$ has full rank for all natural numbers $n$ and divisors $s$ of $n$. Moreover, $\det(M(n,s))=\prod_{d \in D(n,s)} \varphi(d)$, where $\varphi$ is the Euler totient function.
\end{lemma}
\begin{proof}
	We first prove the case when $s=1$. Consider the matrix $N(n)$ (whose determinant equals $\prod_{d\in D'(n)} \varphi(d)$ by Lemma \ref{gcdmatrixeasy}). Within $N(n)$, the first row and column are constantly $1$, and if we subtract the first column from all subsequent columns we get
			\begin{figure}[H]
		\begin{tikzpicture}[
		every left delimiter/.style={xshift=.5em},
		every right delimiter/.style={xshift=-.5em},
		]
		\node at (-3.1,0) {$\det(N(n))= \det$};
		\matrix [ matrix of math nodes,left delimiter={( },right delimiter={ )},row sep=0.1cm,column sep=0.1cm] (U) { 
	1 &0 &...& 0\\
1& & & \\
\vdots & & M(n,1)& \\
1& &&\\
		};

		\node at (3,0) {$= \det(M(n,1))$.};

		\draw (U-1-1.south west) -- (U-1-4.south east);
		\draw (U-1-1.north east) -- (U-4-1.south east);

		\end{tikzpicture}
	\end{figure}\noindent 
	As $\varphi(1)=1$, this verifies the determinant formula in the case of $s=1$.
	
	We proceed by induction on the number of prime divisors of $s$. Assume that $M(n,s)$ has determinant 
	\[\det(M(n,s))=\prod_{d \in D(n,s)} \varphi(d),\]
	and consider $M(p^rn,p^es)$ with $p\nmid n,s$.
	
	Let $d$ be a divisor of $p^rn$, so $d$ is of the form $p^kd'$ with $p\nmid d'$ and $0\le k \le r$. Then, $d \in D(p^rn,p^es)\iff$ either $k \le e$ and $d'\in D(n,s)$, or $k>e$ and $d' \in D'(n)$. In other words, $D(p^rn,p^es)$ can be partitioned as
	\[D(p^rn,p^es) = \left( \bigcup_{i=0}^e p^i D(n,s)\right) \cup \left( \bigcup_{i={e+1}}^r p^i D'(n)\right).\]
	Call $D_1= \bigcup_{i=0}^e p^i D(n,s)$ and $D_2= \bigcup_{i=p^{e+1}}^r p^i D'(n)$. Simultaneously reorder the rows and columns of $M(p^rn,p^es)$ so that they respect this decomposition. Define $A,B,C$ by
			\begin{figure}[H]
		\begin{tikzpicture}[
		every left delimiter/.style={xshift=.75em},
		every right delimiter/.style={xshift=-.5em},
		]
		\node at (-3,0) {$M(p^rn,p^es)=$};
		\matrix [ matrix of math nodes,left delimiter={( },right delimiter={ )},row sep=0.1cm,column sep=0.1cm] (U) { 
	A &C \\
	C^T \hspace{-3pt} & B\\
		};

		\node at (1,-0.5) {.};
		
		\draw (U-1-1.south west) -- (U-1-2.south east);
		\draw (U-1-2.north west) -- (U-2-2.south west);

		\node[above=0pt  of U-1-1]  {$\hspace{2pt} D_1$}; 	
		\node[above=0pt  of U-1-2]  {$D_2$}; 
		
		\node[left=7pt  of U-1-1]  {$D_1$}; 
		\node[left=7pt  of U-2-1]  {$D_2$}; 
		
		\end{tikzpicture}
	\end{figure}\noindent
	For any two elements $p^{l_1}d_1,p^{l_2}d_2\in D_1$, the corresponding entry of $A$ is given by
	\[\gcd(p^{l_1}d_1,p^{l_2}d_2)-\gcd(p^{l_1}d_1,p^{l_2}d_2,p^es)=p^{\min\{l_1,l_2\}}(\gcd(d_1,d_2)-\gcd(d_1,d_2,s)).\]
	So $A$ is the tensor product
	\[A= \left(\begin{array}{c c c c}
	1 & 1 & ...& 1\\
	1 & p & ...&p\\
	\vdots & \vdots & & \vdots\\
	1 & p & ... & p^{e}\end{array}\right)\ot M(n,s)=  N(p^e)\otimes M(n,s),\]
	which has determinant
	\[ \det(A)=\det(N(p^e))^{|D(n,s)|}\cdot \det(M(n,s))^{e+1}.\]
	By induction and Lemma \ref{gcdmatrixeasy},
	\begin{align*}\det(A)&=\det(N(p^e))^{|D(n,s)|}\cdot \det(M(n,s))^{e+1}\\
	&=\left(\prod_{k=0}^e \varphi(p^k)^{|D(n,s)|}\right)\cdot\left(\prod_{d\in D(n,s)} \varphi(d)^{e+1}\right)\\
	&=\prod_{k=0}^e \left(\varphi(p^k)^{|D(n,s)|}\cdot\prod_{d \in D(n,s)}\varphi(d)\right)\\
	&=\prod_{k=0}^e\prod_{d\in D(n,s)} \varphi(p^k)\varphi(n)\\
	&=\prod_{d \in D_1} \varphi(d).
	\end{align*}
	
	We now row reduce to remove $C^T$. For $e<k \le r$, let $v_{p^kd}$ denote the row vector with entries indexed by $D(n,s)=D_1\cup D_2$ whose $t$\textsuperscript{th} entry is defined by
	\[ (v_{p^kd})_t=\gcd(p^ed,t)- \gcd(p^ed,t,p^es).  \]
	If $d \mid s$, then $\gcd(p^ed,t)=\gcd(p^ed,t,p^es)$ and $v_{p^kd}$ is identically zero. If $d \nmid s$, then $p^ed\in D_1$ and $v_{p^kd}$ is the $(p^ed)$\th row of the matrix $M(n,s)$. In either case, subtracting $v_{p^kd}$ from the $(p^kd)$\th row is an elementary row operation and preserves the rank and determinant.
	
	Call $M'(p^rn,p^es)$ the matrix resulting from performing this reduction for all elements of $D_2$. The entries of the $p^kd$\th row for $p^kd \in D_2$ now satisfy, for $p^{k'}d'\in D_1$,
	\begin{align*}M'(p^rn,p^es)_{p^kd,p^{k'}d'}&=\gcd(p^kd,p^{k'}d')-\gcd(p^kd,p^{k'}d',p^es)-\gcd(p^ed,p^{k'}d')+\gcd(p^ed,p^{k'}d',p^es)\\
	&=p^{k'}\gcd(d,d')-p^{k'}\gcd(d,d',s)-p^{k'}\gcd(d,d')+p^{k'}\gcd(d,d',s)\\
	&=0,\end{align*}
	and for $p^{k'}d'\in D_2$,
	\begin{align*}
	M'(p^rn,p^es)_{p^kd,p^{k'}d'}&=\gcd(p^kd,p^{k'}d')-\gcd(p^kd,p^{k'}d',p^es)-\gcd(p^ed,p^{k'}d')+\gcd(p^ed,p^{k'}d',p^es)\\
	&=p^{\min\{k,k'\}}\gcd(d,d')-p^e\gcd(d,d',e)-p^e\gcd(d,d')+p^e\gcd(d,d',s)\\
	&=(p^{\min\{k,k'\}}-p^e)\gcd(d,d').
	\end{align*}
	Therefore, the row reduction results in a matrix of the form
				\begin{figure}[H]
		\begin{tikzpicture}[
		every left delimiter/.style={xshift=.75em},
		every right delimiter/.style={xshift=-.75em},
		]
		\node at (-3,0) {$M'(p^rn,p^es)=$};
		\matrix [ matrix of math nodes,left delimiter={( },right delimiter={ )},row sep=0.1cm,column sep=0.1cm] (U) { 
			A &C \\
			0 \hspace{-3pt} & B'\\
		};

		\node at (1,-0.5) {.};
		
		\draw (U-1-1.south west) -- (U-1-2.south east);
		\draw (0,0.6) -- (0,-0.6);

		\node[above=0pt  of U-1-1]  {$\hspace{2pt} D_1$}; 	
		\node[above=0pt  of U-1-2]  {$D_2$}; 
		
		\node[left=4pt  of U-1-1]  {$D_1$}; 
		\node[left=7pt  of U-2-1]  {$D_2$}; 
		
		\end{tikzpicture}
	\end{figure}\noindent
	where
	\begin{align*}B'&=N(n)\otimes \left(\begin{array}{c c c c}
	p^{e+1}-p^e & p^{e+1} -p^e& ...& p^{e+1}-p^e\\
	p^{e+1}-p^e & p^{e+2} -p^e & ...&p^{e+2} -p^e\\
	\vdots & \vdots & & \vdots\\
	p^{e+1}-p^e  & p^{e+2}-p^e  & ... & p^{r}-p^e\end{array}\right)\\
	&=N(n)\ot M(p^{r},p^e)\\
	& =N(n) \ot p^{e}M(p^{r-e},1).\end{align*}
Since
	\[ \det(M(p^rn,p^es))=\det(A)\cdot \det(B'),\]
	to complete the proof we must show that $\det(B')=\prod_{d \in D_2} \varphi(d)$. Indeed,
		\begin{align*}
 \det(B')&=\det(N(n))^{r-e}\cdot \det(M(p^r,p^e))^{|D'(n)|}\\
	&=\prod_{k=e+1}^r \varphi(p^k)^{|D'(n)|}\det(N(n))\\
	&=\left(\prod_{k=e+1}^r \varphi(p^k)^{|D'(n)|}\right)\cdot \left(\prod_{d \in D'(n)} \varphi(d)\right)\\
	&=.\prod_{k=e+1}^r \prod_{d \in D'(n)}\varphi(p^kd)\\
	&=\prod_{d \in D_2} \varphi(d).
	\end{align*}
	So we find
	\[\det(M(p^kn,p^es))=\left(\prod_{d \in D_1} \varphi(d) \right)\left( \prod_{d \in D_2} \varphi(d) \right)=\prod_{d \in D} \varphi(d). \]
	This completes the proof of the determinant formula of $M(a,b)$ by induction on the number of prime factors of $b$.
\end{proof}

\subsection{Structure of $C_{p^k}\rtimes C_n$}\label{proofofndegcyc}

We now perform an explicit calculation for $p$-hypo-elementary groups before deducing Theorem \ref{tspairingnondeg}.

\begin{lemma}\label{doublecosets}
	Let $G$ be of the form $C_{p^r}\rti C_n$ with $p \nmid n$. Further, let $S$ denote the kernel of the action of $C_n$ on $C_{p^r}$. Then, for any two subgroups $H',K'\le G$ of the form $H'= C_{p^e}\rti H, K' = C_{p^f}\rti K$ with $H,K\le C_n\le G$, as elements of the Burnside ring $B(K')$,
	\[\coprod_{g \in H' \bac G /K' }\hspace{-5pt}[H'^{g} \cap K']=\frac{|C_n||H\cap K|}{|H||K|} [H'\cap K']+\frac{p^{r-\max\{e,f\}}|C_n||H\cap K \cap S| }{|H||K|}[H'\cap K' \cap (C_{p^r}\ti S)]  .\]
\end{lemma}

\begin{proof}
	First assume $e=f=0$. Elements of $H\ba G/K$ are in bijection with $H$-orbits of cosets $gK$. For such $G$, a set of coset representatives of $G/K$ is given by elements $\sigma\tau_i$, where $\sigma \in C_{p^r}$ and $\{ \tau_i\}$ are a set of coset representatives of $C_n/K$. The stabilizer of a right coset $gK$ under the action of $H$ is given by
	\[\Stab_H(gK)=H \cap gKg^{-1}.\]
	Using that $C_n$ is abelian, for $k\in K$,
	\begin{align*}
	(\sigma \tau_i)k(\sigma\tau_i)^{-1}&=\sigma \tau_i k \tau_i^{-1}\sigma^{-1}=\sigma k \sigma^{-1}\\
	&=\sigma k \sigma^{-1} k^{-1} k=\sigma\phi(k)(\sigma^{-1})k,
	\end{align*}
	where $\phi \colon C_n \to \Aut(C_{p^r})$ denotes the action of conjugation. Since the prime to $p$-part of $\Aut(C_{p^r})$ equals that of $\Aut(C_{p^e})$ for any non-trivial $C_{p^e} \le C_{p^r}$, $k$ acts trivially on $\sigma\neq e$ if and only if $k\in S$. Thus,
	\[\sigma\phi(k)(\sigma^{-1})k \in H \iff k \in H \text{ and } k \in  \begin{cases}K & \text{if } \sigma = e\\
	K \cap S& \text{if } \sigma \neq e\end{cases}.\]
	In particular,
	\[\Stab_H(gK) = \begin{cases}H\cap K & \text{if } g \in C_n\\
	H \cap K \cap S& \text{if } g \not \in C_n\end{cases}.\]
	By orbit--stabiliser theorem, there are $\frac{|C_n||H\cap K|}{|H||K|}$ double cosets $HgK$ of length $\frac{|H||K|}{|H\cap K|}$ and \linebreak $(p^r-1)\linebreak[2]\frac{|C_n||H\cap K \cap S|}{|H| |K|}$ double cosets of length $\frac{|H||K|}{|H\cap K\cap S|}$. Furthermore, as $H$ has a unique subgroup of each order
	\[H^g \cap K =\begin{cases} H\cap K &\text{if } g \in C_n\\
	H\cap K \cap S &\text{else}\end{cases},\]
	and applying Mackey's formula \eqref{mackeyforgsets} p\pageref{mackeyforgsets} gives the desired formula in this case.
		
	Now let $e,f \ge 0$. We first calculate the order of $H'\bac G/K'$. As all $p$-subgroups of $G$ are normal, there are canonical bijections
	\begin{equation*}(C_{p^e}\rti H)\backslash G /(C_{p^f}\rti K)\leftrightarrow H \backslash G/(C_{p^{\max\{e,f\}}}\rti K)\leftrightarrow
	H \backslash ((C_{p^r}/C_{p^{\max\{e,f\}}})\rti C_n)/K.\label{canbij}\end{equation*}
	So, from the first part, we find there are $\frac{|C_n||H\cap K|}{|H||K|}$ double cosets of length $\frac{|H||K|}{|H\cap K|}p^{\max\{e,f\}}$ and $(p^{\max\{e,f\}}-1)\frac{|C_n||H\cap K \cap S|}{|H| |K|}$ double cosets of length $\frac{|H||K|}{|H\cap K\cap S|}p^{\max\{e,f\}}$. Taking preimages,
	\[H'^{g} \cap K' =\begin{cases} H'\cap K' &\text{if } g \in C_n\\
	H'\cap K' \cap(C_{p^r}\ti S) &\text{else}\end{cases}.\]
	
	Therefore, indeed
	\[\coprod_{g \in H' \bac G /K' }\hspace{-5pt}[H'^{g} \cap K']=\frac{|C_n||H\cap K|}{|H||K|} [H'\cap K']+\frac{p^{r-\max\{e,f\}}|C_n||H\cap K \cap S| }{|H||K|}[H'\cap K' \cap (C_{p^r}\ti S)]  .\vspace{-15pt}\]
		
\end{proof}
\begin{notat}\label{centraliser}
	Given a group $G$ and subgroup $H\le G$, we denote by $N_G(H)$ the normaliser of $H$ in $G$ and by $Z_G(H)$ its centraliser. 
\end{notat}
\begin{proof}[Proof of Theorem \ref{tspairingnondeg}] Lemma \ref{permredtophypo} states that the permutation pairing for $G$ is non-degenerate if the pairing is non-degenerate for all $p$-hypo-elementary subgroups. So we shall assume that $G$ is $p$-hypo-elementary, i.e.\ $G\cong C_{p^r}\rti C_n$ with $p \nmid n$.
	
	 For notational convenience, we make a fixed choice of subgroup of $G$ isomorphic to $C_n$, which we also denote by $C_n$. Let $S$ denote the kernel of the map $C_n\to \Aut(C_{p^r})$ defining the semi-direct product. Note that $S$ is also the kernel of the map $C_n\to \Aut(C_{p^k})$ for all $1\le k \le r$. Up to conjugacy, any subgroup of $G$ is of the form $C_{p^k}\rti L$, with $L$ contained in the fixed choice of $C_n$. Moreover, such a subgroup is cyclic and normal in $G$ if and only if $L\le S$.

	Let $H',K'$ be non-cyclic subgroups of $G$. We may assume, by replacing $H',K'$ with conjugate subgroups if necessary, that $H'= C_{p^e}\rti H, K'= C_{p^f}\rti K$ with $H,K\le C_n$. We first calculate $\la H',K'\ra_\perm=v_p(C_{\theta_{H'}}(\one _{K'}\up^G))=v_p(C_{\theta_{H'}\down_{K'}}(\one_{K'}))$. The decomposition of $\theta_{H'}\down_{K'}$ matches that of its leading term (Lemma \ref{mackeyforartrel}), so applying Lemma \ref{doublecosets} we find
		\[ \theta_{H'} \up^G\down _{K'}= \left( \frac{|C_n||H\cap K|}{|H||K|}\right)  \cdot \theta_{H'\cap K'}\up^{K'}+\left( \frac{p^{r-\max\{e,f\}}|C_n||H\cap K \cap S| }{|H||K|}\right) \cdot \theta_{H'\cap K' \cap (C_{p^r}\ti S)}\up^{K'}.  \]
		But $H'\cap K' \cap (C_{p^r}\ti S)$ is cyclic (so that $\theta_{H'\cap K' \cap (C_{p^r}\ti S)}=0$) and we find
		\[ v_p(C_{\theta_{H'}}    (\one_{K'}\up^G) ) = \frac{|C_n||H\cap K|}{|H||K|}v_p(C_{\theta_{H'\cap K'}}(\one_{H'\cap K'})).\]
		
		Let $L'$ be an arbitrary non-cyclic subgroup of the form $C_{p^\ell}\rti L$ with $L \le C_n$. Directly applying \eqref{oneform} to the formula of Example \ref{artinrelcyclic}, or by looking ahead to Example \ref{phyporegconstcalc}, we find that
	\begin{align*}v_p(C_{\theta_{L'}}(\one _{L'}))&=-\ell(1- \frac{|Z_{L'}(C_{p^\ell})|}{|N_{L'}(C_{p^\ell})|})\\
	&=-\ell(1-\frac{|L\cap S|}{|L|}).
	\end{align*}
	Concluding our calculation of $\la H',K'\ra _{\perm}$, we find
		\begin{align} \la H',K'\ra_\perm
	&= \frac{|C_n||H\cap K|}{|H||K|}v_p(C_{\theta_{H'\cap K'}}(\one_{H'\cap K'})) \nonumber\\
	&= \frac{|C_n||H\cap K|}{|H||K|}\min\{e,f\}\left( \frac{|H\cap K \cap S|}{|H\cap K|}-1 \right) \nonumber\\
	&=\frac{|C_n|\min\{e,f\}}{|H||K|}(|H\cap K \cap S|-|H\cap K|). \label{permpairingformula}
	\end{align}

	Let $T$ be the matrix representing the pairing $\la \ \, , \ \ra_\perm$ with respect to the basis of $P(G)$ given by non-cyclic $p$-hypo-elementary subgroups ordered (totally in our case) by size. After a non-zero scaling of the rows and columns of $T$, we obtain a matrix $T'$ with $(H',K')$\th entry
	\[T'_{H',K'}= \min\{e,f\}(|H\cap K \cap S|-|H\cap K |). \]
	Note $T'$ remains symmetric and has the same rank as $T$. Since $C_n$ is cyclic, $|H\cap K \cap S|=\gcd(|H|,|K|,|S|)$ and $|H \cap K | =\gcd(|H|,|K|)$. Thus, $T'$ is the matrix with entries
	\begin{align} 
	T'_{H',K'}&=\min\{e,f\}\left(\gcd(|H|,|K|,|S|)-\gcd(|H|,|K|)\right). \label{sofat}
	\end{align}
	Let $M(m,l)$ be as in Notation \ref{matrixnotat}. If $Q(d)$ denotes the $d\ti d$ matrix with $Q_{i,j}=\min\{i,j\}$, then, by \eqref{sofat}, we may simultaneously permute the rows and columns of $T'$ to get
	\[T' \sim -Q(r) \ot M(n,s),\]
	where $|S|=s$. As $Q(r)$ is manifestly of full rank and Lemma \ref{gcdmatinvertible} states that $M(n,s)$ is also, so the same is true for $T$ and the permutation pairing for $G$ is non-degenerate.
\end{proof}
\begin{example}
	Let $G= C_7\rti C_{12}$. A set of representatives of the non-cyclic conjugacy classes of $G$ is given by
	\[ S:= \{ C_7\rti C_3, C_7\rti C_4, C_7 \rti C_6, C_7\rti C_{12}. \} \] 
	Applying \eqref{permpairingformula}, the matrix $T$ representing the permutation pairing with respect to the basis given by $S$ is given by 
	\[ \left(\begin{array}{c c c c}
-8/3 & 0 & -4/3&-2/3 \\
0 & -3/2 & 0& -1/2 \\
-4/3 & 0 &-4/3 & -2/3\\
-2/3 & -1/2 & -2/3 & -5/6
\end{array}\right).\]
In the notation of the proof of Theorem \ref{tspairingnondeg}, $n=12$ and $s=2$. After rescaling the rows and columns of $T$ as in the proof, we obtain the matrix $M(12,2)$ of Example \ref{12,2}.
\end{example}

\section{Non-vanishing of the Artin regulator constant}\label{nonvanartin}
 In this section, we prove:
\begin{thm}\label{nonvanishing}For any finite group $G$ and prime $p$, $v_p(C_{\theta_G}(\one _G))\neq 0$ if and only if $G$ contains a non-cyclic $p$-hypo-elementary subgroup. If $G$ does contain a non-cyclic $p$-hypo-elementary subgroup then $v_p(C_{\theta_G}(\one_G))\le -p/|G|$. Here, $\one_G$ denotes the trivial $\z_p[G]$-module.
\end{thm}
The method of proof is of explicit group theoretic natured and is disjoint to that of Section \ref{regconstinv}. Moreover, Sections \ref{arblat} and \ref{examples} have no dependency on this section.
\begin{remark}\label{relp}
	The forward direction of Theorem \ref{nonvanishing} is formal: If $G$ contains no non-cyclic $p$-hypo-elementary groups then all characteristic zero relations are relations in characteristic $p$ (see Lemma \ref{quotbasis}). But the regulator constant of a characteristic $p$ relation has trivial valuation at $p$ when evaluated at any lattice (Lemma \ref{prelnopterms}). 
\end{remark}
\begin{remark}
	Let $G$ be a $p$-hypo-elementary group. Then, in terms of the permutation pairing of Construction \ref{permpairingconstr}, the theorem asserts that every entry in the row and column corresponding to $G$ is strictly negative. By Lemma \ref{reductiontophypo}, the regulator constant pairing is non-degenerate whenever each $p$-hypo-elementary subgroup of $G$ contains only cyclic proper subgroups, e.g.\ $G=S_4$. Under the same hypothesis, permutation modules over $\z_p$ are determined by extension of scalars to $\q_p$ and regulator constants (Lemma \ref{tsdet}). 
\end{remark}
\begin{corol}\label{singlenonvan}
	For any finite group $G$, as a function on $\z[G]$-modules, the regulator constant associated to the Artin relation $\theta_H$ vanishes identically if and only if $H$ is cyclic.
\end{corol}
\begin{proof}
	For cyclic $H$, $\theta_H=0$ so its regulator constant is trivial. For the converse, we first show:
	
\begin{claim} A finite group $K$ is cyclic if and only if all its $\ell$-hypo-elementary subgroups are cyclic for all $\ell$.\end{claim}
	\begin{proof}[Proof of Claim.]
		Suppose $K$ is a group for which all its $\ell$-hypo-elementary subgroups are cyclic for all $\ell$. Then, as the Sylow subgroups must be cyclic, the normaliser of every Sylow subgroup must be equal to its centraliser. Burnside's normal $p$-complement theorem then ensures that every Sylow $p$-subgroup normalises every Sylow $\ell$-subgroup for $\ell\neq p$. As a result, $K$ is a direct product of its (cyclic) Sylow subgroups for different $p$, and thus $K$ is cyclic.
	\end{proof}
	 Now suppose $T\le G$ is non-cyclic. By the claim, $T$ has a non-cyclic $\ell$-hypo-elementary subgroup $L$ for some $\ell$. Then,
	 \begin{align*} 0\! &\overset{\ \textrm{ \ref{nonvanishing}}}{\ \ \, >} \ \ \, v_{\ell}(C_{\theta_L}(\one_{\z_{\ell},L}))\\
	 &\overset{\textrm{ \ref{mackeyforartrel} {\it i)}}}{=} v_{\ell}(C_{\theta_T\down_L}(\one_{\z_{\ell},L}))\\
	 &\overset{\textrm{\ref{regcontoolkit} {\it iv)}}}{=} v_{\ell}(C_{\theta_T}(\one_{\z_{\ell},L}\up^T))\\
	 & \overset{\textrm{\ref{regcontoolkit} {\it vi)}}}{=} v_{\ell}(C_{\theta_T}(\one_{\z,L}\up^T)).\qedhere
	    \end{align*}
\end{proof}
\begin{remark}
	By symmetry (Lemma \ref{sym}), we find that a permutation module $\one \up^G_H$ is trivial under all regulator constants if and only if $H$ is cyclic.
\end{remark}
\subsection{Explicit Artin induction}
 The proof of Theorem \ref{nonvanishing} is made possible by Brauer's formula for explicit Artin induction. 

\begin{notat}
Let $\mu(n)$ denote the M\"{o}bius function of a natural number $n$,
\[\mu(n)= \begin{cases}
(-1)^r & \text{if $n$ is squarefree and has $r$ distinct prime factors}\\
0 & \text{if $n$ is not squarefree}
\end{cases} . \]
Note that $\mu(1)=1$.
\end{notat}
\begin{lemma}[Brauer, {\cite[Thm.\ 2.1.3]{SnaithExplicitBrauer}}] \label{expart}
	If $G$ is any finite group and $\theta_G= [G]-\sum_{\cyc(G)}\alpha_H [H]$ is its Artin relation, then
	\[\alpha_H= \frac{1}{|N_G(H): H|}\sum_{C\ge H}\mu(|C : H|).\]
	Here the sum runs over all cyclic overgroups of $H$ (not just up to conjugacy).
\end{lemma}
\begin{lemma}\label{denom}\vspace{-\topsep}
	Let $G$ be a $p$-hypo-elementary group and $\theta_G= [G]-\sum_{\substack{H\le_GG \\H\text{ cyclic}}}\alpha_H[H]$. Then $\alpha_H \in \frac{p}{|G|}\cdot \z$. 
\end{lemma}
\begin{proof}
	Let $G=P\rti C$ be non-cyclic and $H\le G$ cyclic. By explicit Artin induction, $\alpha_H \in \frac{1}{|N_G(H): H|}\cdot \z$, so there is only anything to prove when $H$ is of order coprime to $p$ and $H$ is normalised by $P$ (and so by $G$). Such an $H$ must therefore lie in the kernel $S$ of the action of $C$ on $P$.
	
	Let $q$ be the quotient map	$ q \colon G \to G/H$.
	Then, a subgroup $ K\le G$ is cyclic if and only if $q(K)$ is. So $q$ defines an index preserving bijection between cyclic subgroups of $G$ containing $H$ and cyclic subgroups of $G/H$. As such, we may assume that $H=\{1\}$.

	We shall show that the contributions to $\sum_{K\text{cyclic}}\mu(|K|)$ from cyclic subgroups of order coprime to $p$, and of order divisible by $p$ exactly once, cancel (recall that $\mu(|K|)$ vanishes for all other $K$). Let $K$ be a cyclic subgroup of $G$ of order coprime to $p$. We split into two cases: First assume $K$ is normal. Any cyclic group containing $K$ with index $p$ is of the form $C_p\ti K$ for some $C_p\le P$. By (the general form of) Sylow's theorems there are $1$ $(\text{\rm mod } p)$ such choices. Since $\mu(|C_p\ti K|)=-\mu(|K|)$ the contributions of $K$ and its overgroups cancel modulo $p$.

	Now assume that $K$ is not normal. In particular, $K$ is not normalised by $P$ and there are no cyclic subgroups isomorphic to $C_p\ti K$. As $P$ acts transitively on the non-singleton set of conjugates of $K$, orbit-stabiliser shows that the number of subgroups of $G$ isomorphic to $K$ is $0$ $(\text{\rm mod }p )$. We have exhausted all cyclic subgroups and thus $p$ divides $\sum_{\substack{C\le G\\ C\text{cyclic}}}\mu(|C|)$ and $\alpha_H \in \frac{p}{|G|}\cdot \z$.
\end{proof}
\begin{corol}\label{denoms}
	For any non-cyclic $p$-hypo-elementary group $G$ and module $M$, $v_p(C_{\theta_G}(M))\in \frac{p}{|G|}\cdot \z$. More generally, for any finite group $G$, given subgroups $H,K$ and a $K$ module $M$, $v_p(C_{\theta_H}(M\up^G_K))\in\frac{p}{\gcd\{|H|,|K|\}}\cdot \z$. 
\end{corol}
\begin{proof}
By definition the valuations of regulator constants of integral Brauer relations lie in $\z$, so the first statement follows from the lemma and \ref{regcontoolkit}\,{\it iii)}. For the second, the formalism of Lemma \ref{regcontoolkit} and Mackey's formula gives
\begin{align*}v_p(C_{\theta_H\up^G_H}(M\up^G_K))&=v_p(C_{\theta_H}(M\up^G_K\down_H))\\ 
&=\sum_{g \in H\ba G/K}\hspace{-5pt}v_p(C_{\theta_H}(M^g\down_{K^g\cap H }\up ^H))\\
&=\sum_{g \in H\ba G/K}\hspace{-5pt}v_p(C_{\theta_H\down_{K^g\cap H}}(M^g\down_{K^g\cap H })). \end{align*}
But applying the first statement, each term of the sum lies in $\frac{p}{\gcd\{|H|,|K|\}}\cdot \z$.
\end{proof}
We now look to use explicit Artin induction to provide a formula for $v_p(C_{\theta_G}(\one_G))$.
\begin{notat}
	Recall that if two subgroups $H_1$, $H_2$ of $G$ are conjugate, then $[H_1]$ and $[H_2]$ are isomorphic as $G$-sets. To make the Artin relation slightly more canonical, instead of writing 
	\[\theta_G = [G]-\sum_{H\le_G G} \alpha_H[H],\]
	we can choose to write $\theta_G$ uniquely as
	\[\theta_G = [G] -\sum_{H\le G} \alpha'_H [H],\]
	subject to the stipulation that $\alpha'_{H_1}=\alpha'_{H_2}$ for conjugate $H_1,H_2$. Then $\alpha'_H=\frac{1}{|G : N_G(H)|}\cdot \alpha_H$, the $\alpha_H'$ are unique and the two notational choices denote identical elements of $B(G)$.
\end{notat}
\begin{notat}
	Fix a single prime $p$ for the remainder of this section. Let $\pp(G,k)$ denote the number of elements of a given finite group $G$ whose order is divisible by $p^k$.	
\end{notat}
\begin{lemma}\label{regconstfirstform}\vspace{-\topsep}\vspace{3pt}
	For any group $G$ and prime $p$, if $\theta_G$ denotes the Artin relation, then
		\begin{equation}
	v_p(C_{\theta_G}(\one _G))=-v_p(|G|)+\frac{1}{|G|}\sum_{g \in G}v_p(|g|)+\frac{\mathcal{P}(G,1)}{|G| \cdot (p-1)}  .\label{explicitformula}
	\end{equation}
\end{lemma}
\begin{proof}
	Running over all cyclic subgroups rather than their conjugacy classes, explicit Artin induction gives that
	\[\theta_G = [G]-\sum_{\substack{H\le G\\H \text{-cyclic}}}[H]\cdot \frac{1}{|G : H|}\sum_{\substack{C\ge H\\ C \text{-cyclic}}}\mu(|C : H|).\]
	Applying the formula \eqref{oneform} for regulator constants at the trivial module we find that 
	\begin{align*}
		v_p(C_{\theta_G}(\one _G))&=-v_p(|G|)+\sum_{H\le G}\frac{v_p(|H|)}{|G: H|}\sum_{C \ge H}\mu(|C : H|), \\  
		\intertext{where from now on it is assumed that sums run only over all cyclic subgroups or overgroups. Changing the order of summation,}
			v_p(C_{\theta_G}(\one _G))&=-v_p(|G|)+\sum_{C\le G}\sum_{H\le C}\frac{v_p(|H|)}{|G: H|}\mu(|C : H|)   \\
		&=-v_p(|G|)+\sum_{\substack{C\le G\\ p \mid |C| }}\sum_{H\le C}\frac{v_p(|H|)}{|G: H|}\mu(|C : H|)   \\
		\intertext{as only subgroups $C$ for which $p$ divides $|C|$ make any contribution. Within the second sum, by definition of the M\"{o}bius function, only the subgroups of squarefree index contribute. We separate into the sums over the subgroups $H$ of $C$ of index divisible by $p$, and subgroups $H$ of index not divisible by $p$. There is a bijection between these two sets given by sending a subgroup $H$ of index not divisible by $p$ to $pH$. Thus,}
			v_p(C_{\theta_G}(\one _G))&=-v_p(|G|)+\sum_{\substack{C\le G\\ p \mid |C| }}\sum_{\substack{H\le C\\ p \nmid |C: H|}}\frac{v_p(|H|)}{|G: H|}\mu(|C : H|) +\sum_{\substack{C\le G\\ p \mid |C| }}\sum_{\substack{H\le C\\ p \nmid |C: H|}}\frac{v_p(|pH|)}{|G: pH|}\mu(|C : pH|)  \\
		&=-v_p(|G|)+\sum_{\substack{C\le G\\ p \mid |C| }}\sum_{\substack{H\le C\\ p \nmid |C: H|}}\frac{v_p(|H|)}{|G: H|}\mu(|C : H|) -\sum_{\substack{C\le G\\ p \mid |C| }}\sum_{\substack{H\le C\\ p \nmid |C: H|}}\frac{v_p(|H|)-1}{p|G: H|}\mu(|C : H|)  \\
		&=-v_p(|G|)+\sum_{\substack{C\le G\\ p \mid |C| }}\sum_{\substack{H\le C\\ p \nmid |C: H|}}\left(v_p(|H|)\cdot \frac{p-1}{p}\cdot\frac{\mu(|C : H|)}{|G: H|} +\frac{\mu(|C : H|)}{p|G: H|} \right)\\
		&=-v_p(|G|)+\underbrace{\sum_{\substack{C\le G\\ p \mid |C| }}\sum_{\substack{H\le C\\ p \nmid |C: H|}}v_p(|H|)\cdot \frac{p-1}{p}\cdot\frac{\mu(|C : H|)}{|G: H|}}_{(\dagger)} +\underbrace{\sum_{\substack{C\le G\\ p \mid |C| }}\sum_{\substack{H\le C\\ p \nmid |C: H|}}\frac{\mu(|C : H|)}{p|G: H|}}_{(\star)}
	\end{align*}
	We claim that $(\star)$ is equal to $\frac{\mathcal{P}(G,1)}{|G| \cdot (p-1)}$ and $(\dagger)$ is equal to $\frac{1}{|G|}\sum_{g \in G}v_p(|g|)$. To see this suppose that $f \colon G \to \co$ is any map of sets which is constant on elements $g\in G$ for which $v_p(|g|)$ is equal and for which $f(g)=0$ when $v_p(|g|)=0$. In this case, we have that
	\[\sum_{\substack{C\le G\\ p \mid |C| }}\sum_{\substack{H\le C\\ p \nmid |C : H|}}\frac{p-1}{p}\frac{\mu(|C: H|)}{|G: H|}f(h)=\frac{1}{|G|}\sum_{g \in G} f(g),\]
	where on the left hand side $h$ denotes any generator of $H$. This follows from the fact that for a cyclic group $C$
	\[\sum_{{H \le C}} \mu(|C: H|) |H| = |\{\textrm{generators of C}\}|. \]
	Setting $f(g)=v_p(|g|)$ gives
	\[\sum_{\substack{C\le G\\ p \mid |C| }}\sum_{\substack{H\le C\\ p \nmid |C: H|}}v_p(|H|)\cdot \frac{p-1}{p}\cdot\frac{\mu(|C : H|)}{|G: H|}=\frac{{1}}{|G|}\sum_{g\in G} v_p(|g|),  \]
	whilst taking $f(g)=\begin{cases}
	\frac{1}{p-1} & p \mid  |g|\\
	0 & p \nmid |g|
	\end{cases}$ shows that
	\begin{align*}\sum_{\substack{C\le G\\ p \mid |C| }}\sum_{\substack{H\le C\\ p \nmid |C: H|}}\frac{\mu(|C : H|)}{p|G: H|}&=\frac{1}{|G|}\sum_{\substack{g\in G\\v_p(|g|)\ge 1}}\frac{1}{p-1}\\
		&= \frac{\mathcal{P}(G,1)}{|G| \cdot (p-1)}.
	\end{align*}
	In conclusion,
	\[v_p(C_{\theta_G}(\one _G))=-v_p(|G|)+\frac{1}{|G|}\sum_{g \in G}v_p(|g|)+\frac{\mathcal{P}(G,1)}{|G| \cdot (p-1)}  .\]
\end{proof}
We shall see that the value of \eqref{explicitformula} is less than or equal to zero for all groups $G$. Thus, Theorem \ref{nonvanishing} gives a numerical characterisation of groups for which all $p$-hypo-elementary subgroups are cyclic:
\begin{corol}
	Let $G$ be any finite group and $p$ a prime. Then, $G$ contains no non-cyclic $p$-hypo-elementary subgroups if and only if
	\begin{align*} \frac{1}{|G|}\sum_{g \in G}v_p(|g|)+\frac{\mathcal{P}(G,1)}{|G| \cdot (p-1)}= v_p(|G|). \end{align*}
\end{corol}
The reverse direction whilst a consequence of the argument given in Remark \ref{relp} is already somewhat non-obvious.
\begin{remark}\label{noncycpylow}
	Suppose that $G$ has non-cyclic Sylow $p$-subgroups. Let $d=v_p(|G|)$, as $G$ contains no elements of order $p^d$, for any $g\in G$ $v_p(|g|)\le d-1$ and we may crudely bound
	\begin{align*}
		\frac{1}{|G|} \sum_{g \in G}v_p(|g|)+\frac{\mathcal{P}(G,1)}{|G| \cdot (p-1)}&\le\frac{1}{|G|}\sum_{g \in G} \left(d-1 +\frac{1}{p-1}\right)\\
		&< d.
	\end{align*}
	Applying \eqref{explicitformula} gives
	\[v_p(C_{\theta_G}(\one _G ))<0.\]
	The case of cyclic Sylow $p$-subgroups requires considerably more care.
\end{remark}

\subsection{Average $p$-orders of elements of groups with cyclic Sylow subgroups}
In this section, we complete the proof of Theorem \ref{nonvanishing} by explicit calculation of $v_p(C_{\theta_G}(\one_G)))$ for groups with cyclic Sylow $p$-subgroups using \eqref{explicitformula} for groups with cyclic Sylow $p$ subgroups. This requires an explicit calculation of $\pp(G,k)$ in terms of elementary invariants:

\begin{prop}\label{avgformula}
	Let $G$ be any group with cyclic Sylow $p$-subgroups of order $p^r$. Then, for any $1\le k \le r$,
	\[\pp(G,k)= \left(\frac{p^{r-k+1}-1}{p^{r-k+1}}\right)  \frac{|G||Z_G(Q)|}{|N_G(Q)|},\]
	where $Q$ denotes any choice of non-trivial $p$-subgroup of $G$. If $k=0$, then $\pp(G,k)=|G|$ and if $k>r$, then $\pp(G,k)=0$.
\end{prop}
Here, $Z_G(-)$ is as defined in Notation \ref{centraliser}. We split the proof into four intermediate claims. Firstly, the ratio $|N_G(Q)|/|Z_G(Q)|$ is independent of the choice of $Q$:
\begin{claim} {\sc 1}\label{indepofQ}
	If $G$ is any finite group and $p$ a prime such that $G$ has cyclic Sylow $p$-subgroups, then, as $Q$ runs over non-trivial $p$-subgroups, $|N_G(Q)|/|Z_G(Q)|$ is constant.
\end{claim}
\begin{proof}\vspace{-\topsep}\vspace{3pt}
	For such a group all $p$-subgroups of the same order are conjugate. If $Q,Q'$ are conjugate $p$-subgroups their normalisers and centralisers are related by conjugation and so the above ratio is constant. Thus, we need just show that if $P$ is a subgroup of order $p^e$, $e\ge 2$, and $Q$ its unique subgroup of order $p^{e-1}$, then
	\begin{equation}\frac{|N_G(P)|}{|Z_G(P)|}=\frac{|N_G(Q)|}{|Z_G(Q)|}. \label{quotequal} \end{equation}
	First note that $N_G(P)\cap Z_G(Q)=Z_G(P)$. This is because both sides contain $P$, but the coprime to $p$-part of $\Aut(P)$ is canonically isomorphic to the coprime to $p$-part of $\Aut(Q)$ (both are cyclic of order $p-1$). In other words, within $N_G(P)$, to centralise $Q$ is to centralise $P$. As a result, there is an inclusion
	\[N_G(P)/Z_G(P)\hookrightarrow N_G(Q)/Z_G(Q),\]
	and to prove \eqref{quotequal} we must show that $N_G(Q)=N_G(P)Z_G(Q)$. Indeed, as all terms are contained in $N_G(Q)$, we may assume that $Q\unlhd G$. Each choice of subgroup of order $p^e$ (i.e.\ conjugate of $P$) must centralise $Q$, its unique subgroup of order $p^{e-1}$, thus $\bigcup_{g \in G/N_G(P)}P^g\subseteq  Z_G(Q)$. In particular, $Z_G(Q)$ contains a representative of each coset of $G/N_G(P)$ and so $N_G(P)Z_G(Q)=G$. And in general, $N_G(P)Z_G(Q)=N_G(Q)$.
\end{proof}
Next, we show that to prove the formula for fixed $k$ we may reduce to groups with a central $C_{p^k}$ subgroup.

\begin{claim}{\sc 2} For any finite group $G$, prime $p$ and $k\ge 1 $, all elements of $G$ of order divisible by $p^k$ are contained in $\bigcup_{Q}Z_G(Q)$ as $Q$ runs over subgroups of $G$ isomorphic to $C_{p^k}$. Moreover, if $G$ has cyclic Sylow $p$-subgroups, then 
	\[\pp(G,k)=|G : N_G(Q)|\cdot \pp(Z_G(Q),k),\]
	for any choice of $Q\cong C_{p^k}$.
\end{claim}
\begin{proof}
	Let $g \in G$ and $v_p(|g|) \ge k$. Then, $g$ centralises the subgroup of $\la g \ra$ isomorphic to $C_{p^k}$. So $g$ is contained in $\bigcup_{Q}Z_G(Q)$, the union of the centralisers of all $C_{p^k}$-subgroups of $G$. Now let $Q\le G$ be some $C_{p^k}$-subgroup. Since $G$ has cyclic Sylow $p$-subgroups, $Q$ must be the unique $C_{p^k}$-subgroup of $Z_G(Q)$. As a result, if $Q'$ is a distinct $C_{p^k}$-subgroup, then $Z_G(Q)\cap Z_G(Q')$ does not contain any $C_{p^k}$-subgroup, and so $\pp(Z_G(Q)\cap Z_G(Q'),k)=0$. Thus,
	\[\pp(G,k) = \sum_Q \pp(Z_G(Q),k)=|G : N_G(Q)|\cdot  \pp(Z_G(Q),k).\qedhere\]
\end{proof}
As a basis for induction we show:

\begin{claim}{\sc 3}
	Let $G$ be any group and $Q$ a subgroup of order $p$ that is contained in the centre. Then,
	\[\pp(G,1)=\pp(G/Q,1)+\frac{p-1}{p}\cdot |G|. \] 
\end{claim}
\begin{proof}
	Consider the sequence
	\[1 \to Q \to G \overset{q}{\to} G/Q \to 1,\]
	and let $h$ run over elements of $G/Q$. First assume that $h$ has order not divisible by $p$. As $Q \le Z(G)$, the preimage of $\la h \ra$ is isomorphic to $C_p \ti C_{|h|}$ on which $q$ is projection onto the second factor. Thus, $q^{-1}(h)$ contains precisely $p-1$ elements of order $p$.

	Otherwise, if $h$ has order divisible by $p$, then all elements of $q^{-1}(h)$ have order divisible by $p$. As a result
	\[\pp(G,1)=p\cdot \pp(G/Q,1)+(p-1)(|G/Q|-\pp(G/Q,1)),\]
	giving the stated formula.
\end{proof}
The inductive step is given by:
\begin{claim} {\sc 4}
	Let $G$ be any group with cyclic Sylow $p$-subgroups and containing a central subgroup $Q$ of order $p^k$ with $k\ge 2$. Then,
	\[\pp(G,k)=p\cdot \pp(G/\tilde Q,k-1), \] 
	where $\tilde Q\le Q$ denotes the subgroup of order $p$
\end{claim}
\begin{proof}
	Consider the sequence
	\[1 \to \tilde Q \to G \overset{q}{\to} G/\tilde Q\to 1 .\]
	Running over elements $h \in G/\tilde Q$, we find that if $p^k$ divides $|h|$, then all $p$ preimages have order divisible by $p^k$ and conversely if $p^{k-1}\nmid |h|$, then none do.
	
	Now assume that $p^{k-1}$ is the maximal power of $p$ dividing $|h|$. Then, $H:=q^{-1}(\la h \ra)$ is a subgroup of $G$ with Sylow $p$-subgroups of order $p^k$. Thus, $H$ must be of the form $C_{p^k}\ti A$ with $p$ not dividing the order of $A$. Via this description $q$ is the quotient $C_{p^k}\ti A \to C_{p^k}/C_p \ti A$. So, as $h$ has order divisible by $p^{k-1}$, all elements in the fibre of $h$ have order divisible by $p^k$. In conclusion,
	\[ \pp(G,k)=p\cdot \pp(G/\tilde Q,k-1) .\qedhere \] 
\end{proof}

\begin{proof}[Proof of Prop.\ \ref{avgformula}]
	We first show the formula when $k=1$. The formula trivially holds when $r=0$. If $r\ge 1$, we may apply Claim 2 to assume that $G$ contains a central subgroup $Q$ isomorphic to $C_p$. When $r=1$, the formula is given by Claim 3. Now assume $r\ge 2$. We wish to show that
	\[ \pp(G,1)= \left( \frac{p^r-1}{p^r}\right)|G|. \]
	Applying Claim 3 and the inductive hypothesis,
	\begin{align*}
		\pp(G,1)&=\pp(G/Q,1)+\left(\frac{p-1}{p}\right)|G|\\
		&=\left( \frac{p^{r-1}-1}{p^{r-1}}\right)\frac{|G/Q|\cdot |Z_{G/Q}(Q')|}{|N_{G/Q}(Q')|}+\left(\frac{p-1}{p}\right)|G|,\\
		\intertext{where $Q'$ is a choice of $C_p$-subgroup of $G/Q$. Let $P$ denote the preimage of $Q'$ in $G$. Recall, for a chain of subgroups $A\ge B \ge C$ with $C\unlhd A$, then $N_A(B)/C \cong N_{A/C}(B/C)$. Moreover if $C\subseteq Z(A)$, then $Z_A(B)/C=Z_{A/C}(B/C)$. Thus, $|G/Q : N_{G/Q}(Q')|=|G: N_G(P)|$ and $|Z_{G/Q}(Q')|=\frac{1}{p} |Z_G(P)|$. So}
		\pp(G,1)&=\left( \frac{p^{r-1}-1}{p^{r-1}}\right)\frac{|G|\cdot |Z_G(P)|}{|N_G(P)| \cdot p}+\left(\frac{p-1}{p}\right)|G|\\
		&= \left( \frac{p-1}{p}+\frac{p^{r-1}-1}{p^r}\right)|G|\\
		&=\left( \frac{p^{r}-1}{p^{r}}\right) |G|
	\end{align*} 
	as required, where we used the independence asserted in Claim 1 to show
	\[ \frac{|Z_G(P)|}{|N_G(P)|} =\frac{|Z_G(Q)|}{|N_G(Q)|}=1.\]
	Thus, the formula holds when $k=1$.
	
	Now assume $k>1$ and that the formula holds for all groups and indices $\ell < k$. By Claim 2, we are reduced to verifying the formula for groups with a central subgroup $Q$ isomorphic to $C_{p^k}$.
	
	Fix a subgroup $\tilde Q\le Q $ of order $p$. Applying Claim 4,
	\begin{align*} 
		\pp(G,k)&=p\cdot \pp(G/\tilde Q,k-1)\\
		&=p\left( \frac{p^{(r-1)-(k-1)+1}-1}{p^{(r-1)-(k-1)+1}}\right) \frac{|G/\tilde Q|\cdot |G/\tilde Q|}{|G/\tilde Q|}\\
		&= \left( \frac{p^{r-k+1}-1}{p^{r-k+1}}\right) |G| \end{align*}
	which is the required formula. 
\end{proof}

\begin{proof}[Proof of Theorem \ref{nonvanishing}]
	By Remarks \ref{relp} and Corollary \ref{denoms}, we need only prove that if $G$ has a non-cyclic $p$-hypo-elementary subgroup $v_p(C_{\theta_G}(\one_G))<0$. Whilst, by Remark \ref{noncycpylow}, we may assume $G$ has cyclic Sylow $p$-subgroups.
	
	Applying Lemma \ref{regconstfirstform}, we want to show for such $G$ that
	\[\frac{1}{|G|}\sum_{g \in G}v_p(|g|)+\frac{\mathcal{P}(G,1)}{|G| \cdot (p-1)}\le v_p(|G|),\]
	with equality if and only if all $p$-hypo-elementary subgroups of $G$ are cyclic. Proposition \ref{avgformula} shows that if $G$ has Sylow $p$-subgroups of order $p^r$, then
	\begin{align*}\sum_{g \in G}v_p(|g|)&=\sum_{k=1}^r\pp(G,k)= \sum_{k=1}^r  \left(\frac{p^{r-k+1}-1}{p^{r-k+1}}\right)  \frac{|G||Z_G(Q)|}{|N_G(Q)|} 
		\intertext{and}
		\frac{\mathcal{P}(G,1)}{(p-1)}&= \left( \frac{p^r-1}{(p-1)p^r}\right) \cdot \frac{|G||Z_G(Q)|}{|N_G(Q)|},
	\end{align*}
	where $Q$ denotes any choice of subgroup of $G$ isomorphic to $C_{p}$. Thus,
	\begin{align*}\frac{1}{|G|}\sum_{g \in G}v_p(|g|)+\frac{\mathcal{P}(G,1)}{|G| \cdot (p-1)}&=\left(\left(\sum_{i=1}^r\frac{p^{r-i+1}-1}{p^{r-i+1}}  \right) +\frac{p^r-1}{(p-1)p^r}\right) \frac{|Z_G(Q)|}{|N_G(Q)|}\\
		&=\left(\sum_{i=1}^r \frac{p^{r-i+1}-1}{p^{r-i+1}}+\sum_{i=1}^r\frac{1}{p^{r-i+1}}\right)\frac{|Z_G(Q)|}{|N_G(Q)|}\\
		&=r\cdot \frac{|Z_G(Q)|}{|N_G(Q)|}.
	\end{align*}
	So that
	\[ v_p(C_{\theta_G}(\one_G))=-r\cdot \left(1- \frac{|Z_G(Q)|}{|N_G(Q)|}\right) \]
	whenever $G$ has cyclic Sylow $p$-subgroups. Finally, note that a group with cyclic Sylow $p$-subgroups has no non-cyclic $p$-hypo-elementary groups if and only if all subgroups of order $pq$ with $q$ a prime distinct to $p$ are isomorphic to $C_p\ti C_q$. The latter holds if and only if the normaliser of each $C_p$-subgroup is equal to its centraliser. So indeed $v_p(C_{\theta_G}(\one _G))<0 \iff $ $G$ contains a non-cyclic Sylow $p$-subgroup, otherwise it is zero.
\end{proof}
It is worth stating that during the proof we derived the following corollary:
\begin{corol}
	For any finite group $G$ and prime $p$ such that the Sylow $p$-subgroups of $G$ are cyclic,
	\[ v_p(C_{\theta_G}(\one_G))= -r\cdot\left(1- \frac{|Z_G(Q)|}{|N_G(Q)|}\right).\]
	Here $Q$ denotes any choice of non-trivial $p$-subgroup of $G$ unless $p \nmid |G|$ in which case $Q=\{1\}$ and $v_p(C_{\theta_G}(\one_G))=0$.
\end{corol}
When $G$ doesn't have cyclic Sylow $p$-subgroups, we can only say that $v_p(C_{\theta_G}(\one_G))\le -\frac{p}{|G|}$. 
\begin{example}\label{phyporegconstcalc}
	Let $G$ be a $p$-hypo-elementary group with a non-trivial cyclic Sylow $p$-subgroup. Then $G$ is of the form $C_{p^r}\rti C_n$ with $(p,n)=1$. Let $S$ denote the kernel of the map $C_n\to \Aut(C_{p^r})$ defining the semi-direct product and $s=|S|$. Then,
	\[ v_p(C_{\theta_G}(\one _G))=-r(1- \frac{s}{n}),  \]
	as the centraliser of $C_{p^r}$ is $C_{p^r}\ti S\le G$ (the action is trivial) and $C_{p^r}\unlhd G$.
	
	We can also verify this directly. In Example, \ref{artinrelcyclic} we saw that the Artin relation of such a $G$ is given by
	\[\theta_G =[C_{p^r}\rti C_n]-[C_n]+\frac{s}{n}[S]-\frac{s}{n}[C_{p^r}\ti S].\]
	So applying formula \eqref{oneform}
	\begin{align*}C_{\theta_G}(\one_G)&= \frac{\left(\frac{1}{|G|}\right)\cdot\left(\frac{1}{|S|}\right)^\frac{s}{n}}{\left(\frac{1}{|C_n|}\right) \left( \frac{1}{|C_{p^r}\ti S|}\right)^\frac{s}{n}},
		\intertext{so that indeed}
		v_p(C_{\theta_G}(\one_G))&=-r+\frac{s}{n}\cdot r.
	\end{align*}
	
\end{example}
\section{Yakovlev's theorem and the permutation pairing}
If a group $G$ has cyclic Sylow $p$-subgroups, then Theorem \ref{tspairingnondeg} shows that the permutation pairing is non-degenerate. As an application, if $G$ is in addition abelian, dihedral, or more generally satisfies the conditions of Theorem \ref{main}, then we exhibit an explicit list of invariants which determine the isomorphism class of an arbitrary rationally self-dual $\z_p[G]$-lattice. This requires also understanding the theory over $\zp$, which is somewhat less well behaved.
\label{arblat}
\subsection{Trivial source modules}\label{tsmods}
As we shall see, existing work reduces us to dealing with \emph{trivial source modules}, which we now introduce.
\begin{mydef}\label{trivsource}
	For any finite $G$, we say that an $\R[G]$-module has \emph{trivial source}, or is a \emph{trivial source module}, if it is a direct summand of a permutation module. In other sources, trivial source modules may be referred to as any of relatively projective, permutation projective, $p$-permutation or invertible. For $\R$ any ring, we denote the subalgebra of $A(\R[G])$ generated by the trivial source modules by $A(\R[G],\triv)$.
	
	 Our definition is slightly non-standard (cf.\ \cite[Def.\ 3.11.1]{Benson95}). When $\R=\z_p$, it coincides with the usual definition \cite[Lem.\ 3.11.2]{Benson95}, but when $\R=\zp$, due to the failure of Krull-Schmidt \cite{Beneish}, decomposition by vertices fails (see Example \ref{vertex failure}) and what we call an indecomposable trivial source module over $\zp$ need not have source which is trivial. However, in our definition, for $M$ over $\zp$, $M$ is a trivial source module if and only if $M\ot \z_p$ is a trivial source module.
\end{mydef}
\begin{example}
	Let $\R=\zp$ and $G=C_p$. Up to isomorphism, there are 3 indecomposable $\zp[C_p]$-lattices
	\[ \one_G, I_G, \zp[C_p],  \]
	the trivial module, the augmentation ideal of $\zp[C_p]$, and the regular representation \cite[Thm.\ 2.6]{HellerReiner}. The indecomposable trivial source modules are the summands of $\one_{\{1\}} \up^{C_p}\cong \zp[C_p]$ and $ \one_{C_p}\up^{C_p}=\one_{C_p}$. So the trivial source indecomposables are precisely $\one_G$ and $\zp[C_p]$, and $A(\zp[G],\triv)=A(\zp[G],\perm)$ is two dimensional. The same holds for $\R = \z_p$.
\end{example}
\begin{mydef}\label{tsnotat}
	Let $M$ be any $\z_p[G]$-lattice. We define $M_\triv$ to be the submodule generated by all indecomposable trivial source summands of $M$. We call $M_\triv$ the \emph{trivial source part} of $M$ and call the submodule $M_\nt$ generated by the indecomposable summands which are not trivial source summands the \emph{non-trivial source part}. By the Krull-Schmidt property of $\z_p[G]$-modules, we obtain the \emph{trivial source decomposition} $M=M_\triv \op M_\nt$.
\end{mydef}
\begin{remark}
	Over $\zp$, lattices do not admit unique decomposition \cite{Beneish} and there is no general analogue of the trivial source decomposition. In particular, for $M$ over $\zp$, the trivial source decomposition of $M\ot \z_p$ need not be defined over $\zp$. This occurs for the group $C_3\ti C_4$ when $p=3$ (see \cite[Ch.\ 2]{thesis}).
\end{remark}

\subsection{Yakovlev's result}\label{yakres}
Now assume that $G$ has cyclic Sylow $p$-subgroups.
\begin{notat}
	Let $P$ be a choice of Sylow $p$-subgroup of $G$. Let $r$ be such that $P \cong C_{p^r}$, and for $0\le i \le r$, let $P_i\le P$ denote the subgroup of order $p^i$.
\end{notat}
Note that for a $\z_p[G]$-lattice $M$, $H^1(P_i,M)$ is a $N_G(P_i)$-module.

\begin{samepage}
	\begin{thm}[Yakovlev {\rm \cite[Thm.\ 2.1]{Yak}}]\label{yak} Let $G$ be a finite group and $p$ a prime such that $G$ has cyclic Sylow $p$-subgroups. If $M$ is a $\z_p[G]$-lattice, then the isomorphism class of $M_\nt$ is determined by the following diagram,\nopagebreak \vspace{-5pt}
		\begin{figure}[H]
			\begin{tikzcd}
				H^{1}(P_r,M) \ar[shift left=0.6ex]{r}{\res}& H^{1}(P_{r-1},M)\ar[shift left=0.6ex]{r}{\res}\ar[shift left=0.6ex]{l}{\cores} &\ar[shift left=0.6ex]{l}{\cores} \quad ... \quad \ar[shift left=0.6ex]{r}{\res} &  H^{1}(P_{0},M).\ar[shift left=0.6ex]{l}{\cores}
			\end{tikzcd}\noindent\vspace{-10pt}
			\caption{Yakovlev diagram}\label{yakdiag}
		\end{figure}
\end{thm}\end{samepage}
To be precise, when we say ``determined by'' we mean that, if $M'$ is another $\z_p[G]$-lattice for which there are $\z_p[N_G(P_i)]$-module isomorphisms $\kappa_i : H^{1}(P_{i},M) \to  H^{1}(P_{i},M') , 0 \le i \le n$ which commute with restriction and corestriction in the above diagram, then $M_\nt\cong M'_\nt$.
\begin{constr}
	Call any diagram of the form\vspace{-5pt} 
	\begin{figure}[H]
		\begin{tikzcd}
			\bullet \ar[shift left=0.6ex]{r}{a_r}& \bullet\ar[shift left=0.6ex]{r}{a_{r-1}}\ar[shift left=0.6ex]{l}{b_r} &\ar[shift left=0.6ex]{l}{b_{r-1}} \quad ... \quad \ar[shift left=0.6ex]{r}{a_1} &  \bullet\ar[shift left=0.6ex]{l}{b_1},
		\end{tikzcd}\vspace{-10pt}
	\end{figure}\noindent
	with the $i$\th term a finite $N_G(P_{r-i+1})$-module and $a_i,b_i$ homomorphisms of abelian groups, a \emph{Yakovlev diagram}. For any $M$, Figure \ref{yakdiag} is of this form and we refer to it as the \emph{Yakovlev diagram of $M$}.
	
	There is an obvious notion of direct sum of such diagrams. Let $\mathcal{C}$ denote the free $\q$-vector space on isomorphism classes of such diagrams subject to identifying addition of diagrams with addition of elements of $\mathcal{C}$. Taking Yakovlev diagrams defines a canonical map
	\[ \mathrm{Yak} \colon   A(\z_p[G])\to \mathcal{C}.  \]
	Yakovlev's theorem is now the assertion that $\ker (\mathrm{Yak})= A(\z_p[G],\triv)$. Yakovlev also gives a converse describing which Yakovlev diagrams arise as the cohomology of $\z_p[G]$-lattices, but we do not need this.

\end{constr}

Recall that $A(\zp[G],\perm)=A(\z_p[G],\perm)$ as subspaces of $A(\z_p[G])$. We are now able to correctly formulate the theorem outlined in the introduction:

\begin{thm}
\label{main}
Let $G$ be any finite group and $p$ a prime such that $G$ has cyclic Sylow $p$-subgroups and for which $A(\z_p[G],\perm) = A(\zp[G],\triv)$. Then, the isomorphism class of any rationally self-dual $\z_p[G]$-lattice $M$ is determined by
\begin{enumerate}[i)]
	\item the isomorphism class of $M\ot \q_p$ as a $\q_p[G]$-module,
	\item the valuations $v_p(C_{\theta_H}(M))$ of the regulator constants of the Artin relations for ${H \in \nchyp_p(G)}$,
	\item the Yakovlev diagram \vspace{-5pt}
	\begin{figure}[H]
		\begin{tikzcd}
		H^{1}(P_r,M) \ar[shift left=0.6ex]{r}{\res}& H^{1}(P_{r-1},M)\ar[shift left=0.6ex]{r}{\res}\ar[shift left=0.6ex]{l}{\cores} &\ar[shift left=0.6ex]{l}{\cores} \quad ... \quad \ar[shift left=0.6ex]{r}{\res} &  H^{1}(P_{0},M).\ar[shift left=0.6ex]{l}{\cores}
		\end{tikzcd}\vspace{-10pt}
	\end{figure}\noindent
\end{enumerate}
\end{thm}
\begin{proof}
	By Theorem \ref{yak} the data of {\it iii)} determines the isomorphism class of $M_{\nt}$. Now fix some trivial source $\z_p[G]$-module $M'$ such that $M_\nt \oplus M'$ is rationally self-dual. Such an $M'$ exists as $M_\triv$ is an example. By linearity of regulator constants and extension of scalars, from {\it i),ii),iii)} we also obtain the regulator constants and isomorphism class of the extension of scalars of $t:=[M]-[M_\nt]-[M' ]$ (note that $t$ is rationally self-dual so that its regulator constants are defined). By construction $t\in \ker(\mathrm{Yak}\colon A(\z_p[G])\to \mathcal{C})= A(\z_p[G],\triv)$. It suffices to show that any rationally self-dual element $t\in A(\z_p[G],\triv)$ is determined by the data of {\it i),ii)}.
	
	Let $M''$ be any trivial source $\z_p[G]$-module such that $t-M''$ has rational character, again such an $M''$ certainly exists. Now, any $\z_p[G]$-lattice $N$ for which $N\ot \q_p$ is defined over $\q$ is the extension of scalars of a $\zp[G]$-lattice \cite[Prop.\ 5.7]{ReinerSurvey}. As a result, any element of $A(\z_p[G])$ with rational character is contained within $A(\zp[G])$. In particular, $t-M''\in A(\z_p[G],\triv)\cap A(\zp[G])$. We claim that this space is precisely $A(\zp[G],\triv)$, so let $W,V$ be trivial source $\z_p[G]$-modules and assume that $[W]-[V]$ has rational character. If $V$ is a summand of some permutation module $T$ with complement $V'$, then $[W\oplus V']-[T]=[W]-[V]$ and $W\oplus V'$ is a trivial source module with rational character. So, by \cite[Prop.\ 5.7]{ReinerSurvey}, $[W]-[V']$ lies in $A(\zp[G],\triv)$.
	
	Thus, $t-M''\in A(\zp[G],\triv)$ and so by assumption lies in $A(\z_p[G],\perm)$. As a result, it is determined by its regulator constants and extension of scalars (Corollary \ref{cyctsdet}). Tracing back, we find that $M$ is determined by the data of {\it i),ii),iii)}.
\end{proof}
Since {\it i),ii),iii)} are isomorphism invariants, two $\z_p[G]$-lattices are isomorphic if and only if {\it i),ii),iii)} are the same for both lattices.
\begin{remark}\label{arbitlat}
	The restriction on being rationally self-dual is a somewhat mild one. For example, if $M_1,M_2$ are any two $\z_p[G]$-lattices, then $M_1\cong M_2$ if and only if, {\it i),iii)} of Theorem \ref{main} coincide for $M_1,M_2$ and
		\begin{enumerate}[i)]
		\item[{\it ii')}] there exists some $\z_p[G]$-lattice $N$ such that $M_1\oplus N$ and $M_2\op N$ are both rationally self-dual, and the valuations $v_p(C_{\theta_H }(M_i\op N) )$ of the regulator constants of the Artin relations of $M_i\oplus N$ are equal for all $H\in \nchyp_p(G)$.
	\end{enumerate}
		Note, it is easy to determine if there exists a $\z_p[G]$-lattice $N$ such that $M_1\oplus N$ and $M_2\oplus N$ are rationally self-dual using {\it i)}. If the  $v_p(C_{\theta_H}(M_i\oplus N))$ are equal for one such $N$, then they are equal for all.
\end{remark}
The condition that $A(\z_p[G],\perm)=A(\zp[G],\triv)$ is investigated in the next subsection. For reference, we shall see that the equality can be checked on restriction to the $p$-hypo-elementary subgroups and that dihedral groups, abelian groups with cyclic Sylow $p$-subgroups and groups of order coprime to $p-1$ all have this property, but $C_p\rti C_{p-1}$ for $p\ge 5$ does not. In Section \ref{D2pfullex}, we provide a worked example of Theorem \ref{main} for dihedral groups of order $2p$ with $p$ odd.
\begin{remark}
	Theorem \ref{main} is sharp in the following sense. If $A(\z_p[G],\perm)\subsetneq  A(\zp[G],\triv)$ or the permutation pairing is degenerate, then rationally self-dual $\z_p[G]$-lattices are not determined by {\it i),ii),iii)}. The first case can be seen by comparing the dimension of $A(\z_p[G],\triv)$ and the maximum number of linear conditions on elements of $A(\z_p[G],\triv)$ we could possibly obtain from {\it i),ii)} using Lemmas \ref{formulafortriv}, \ref{condeqcond} and the formulae of Section \ref{relinp}. In the second case, not even all permutation lattices can be distinguished (see Lemma \ref{tsdet}).
\end{remark}
\subsection{Species and trivial source modules over $\zp$}
Now let $G$ be any finite group. 
\begin{mydef} A \emph{species}\footnote{It is more common to define a species as a ring homomorphism from the trivial source ring over the ring of integers of a sufficiently large extension of $\q_p$, but this is not necessary for our purposes.} is a ring homomorphism $A(\z_p[G],\triv)\to \co$.
\end{mydef}
\begin{example}\label{trspecies}
	For any $g\in G$, $\tr(g \mid -)$ defines a ring homomorphism $A(\z_p[G]) \to \co$ and so also a species.
\end{example}
\begin{mydef}
	For $H$ a subgroup of $G$, we say that an indecomposable trivial source $\z_p[G]$-lattice $M$ has vertex $H$ if $M$ is a direct summand of $\one \up^G_H$ but not of $\one \up^G_{H'}$ for any $H'\lneq H$. The vertices of $M$ form a conjugacy class of subgroups and only $p$-groups appear as vertices \cite[Prop.\ 3.10.2]{Benson95}. For an arbitrary $\z_p[G]$-lattice $M$, we call the summand generated by the indecomposables with vertex $P$ the \emph{vertex $P$} summand of $M$, in this way we obtain a decomposition of $M$ indexed by vertices.
\end{mydef}
\begin{constr} Consider pairs $(P,g)$, where $P\le G$ is a $p$-group and $g$ an element of $N_G(P)$ of order coprime to $p$, up to simultaneous conjugacy. Then $H:=\la P,g\ra$ is $p$-hypo-elementary and to any such pair we may associate a species $t_{(P,g)}$ as follows. Consider the composite
\[
			A(\z_p[G],\triv) \to A(\z_p[H],\triv) \to
			A(\z_p[H/P],\triv).
\]
	Here the first map is restriction. The second map sends a lattice $M$ to its vertex $P$ summand $N$, which, as $M$ is trivial source, is inflated from $H/P$ so can be considered as an $H/P$-module. We define $t_{(P,g)}\colon A(\z_p[G],\triv)\to \co$ to be the postcomposition with $\tr(g \mid -)$, i.e.\ $t_{(P,g)}(M)$ is the trace of $g$ acting on $N$.
\end{constr}
\begin{example}\label{trasspecies}
	For any $g\in G$, the species defined by $\tr(g \mid -)$ is equal to $t_{(P,g^{|g|})}$, where $|P|$ is the Sylow $p$-subgroup of $\la g \ra$.
\end{example}	
The $t_{(P,g)}$ need not be distinct, but all species arise in this way:
\begin{thm}[Conlon]\label{conlon}
	For any finite group $G$ and prime $p$, there is an inclusion
	\[ \prod t_{(P,g)}\colon A(\z_p[G],\triv) \to \prod_{(P,g)} \co.   \] 
\end{thm}
\begin{proof} This is usually stated for the ring of integers $\OO_K$ of a sufficiently large extension $K/\q_p$ (see \cite[Cor.\ 5.5.5]{Benson95}). The stated version then follows as, for $K/\q_p$ and $M,M'$ any $\z_p[G]$-lattices, $M\cong M' \iff (M\ot \OO_K)\cong (M'\ot \OO_K)$ and that the action of Galois ensures vertices are preserved under base change by $\OO_K/\z_p$.
\end{proof}
\begin{remark}\label{regvsspecies}
	It is worth remarking that although species are good invariants of trivial source modules they cannot be combined with Yakovlev diagrams to give results such as Theorem \ref{main}. This is because species cannot be canonically extended beyond $A(\z_p[G],\triv)$. On the other hand, regulator constants are defined for an arbitrary rationally self-dual lattice.
\end{remark}
\begin{lemma}\label{fixedpts}
	For a permutation module $\one_K \up^G$, we find  $t_{(P,g)}(\one_K \up^G)=\# (G/K)^{H}$, where $H=\la P,g\ra$.  
\end{lemma}
\begin{proof}
	By definition $t_{(P,g)}$ is a function of $\one_K \up^G\down_H=\bigoplus_{K\ba G/ H}\one \up^H_{K^g\cap H}$. We claim that only the terms with $K^g\cap H =H$ have non-trivial species. Indeed, if $K^g\cap H \not \ngeq P$, then the vertex $P$ summand of $\one \up^H_{K^g\cap H}$ is zero, whilst if $K^g\cap H \ge P$, then $\one \up^H_{K^g\cap H}$ is all of vertex $P$ and is inflated from the quotient $\la g\ra$. It is then clear that $\tr(g \mid -)$ is zero if and only if $K^g\cap H \neq H$, else it is one.
	
	Finally, the number of elements of $K\ba G/H$ with $K^g\cap H =H$ is precisely the number of elements of $G/K$ fixed under $H$. 
\end{proof}
\begin{example}\label{vertex failure}
	It is not the case that the species of a trivial source lattice $M$ over $\zp$ need take only rational values. This is made possible by the failure of Krull-Schmidt over $\zp$. For example, if $p\ge 5$ and $G=C_p\rti C_{p-1}$ with $C_{p-1}$ acting faithfully, then a trivial source module with non-rational species is constructed as follows (cf.\ \cite{Beneish}). Let $\chi$ denote the inflation of a faithful character of $C_{p-1}$. The trivial source $\z_p[G]$-modules are then the summands of
	\begin{align*} \one\up^G&= \bigoplus_{i=0}^{p-2} \one \up^G_{C_{p-1}}\ot \chi^i,\\
	\one \up ^G_{C_p}&= \bigoplus_{i=0}^{p-2} \chi^i.
	\end{align*}
	Then, 
	\[M:= \left(\bigoplus_{\substack{i=0\\i\neq 1 }}^{p-2} \one \up^G_{C_{p-1}}\ot \chi^i\right) \op \chi\]
	is defined over $\zp$, but as the only summand of vertex $C_p$ is $\chi$, the vertex decomposition is not defined over $\zp$, and the species of $M$ are non-rational, as $t_{(C_p,\tau)}(M)=\chi(\tau)$ is a primitive $(p-1)$\tss{st} root of unity.
	
	In particular, $A(\z_p[G],\perm) \subsetneq A(\zp[G],\triv)$ as the species of permutation modules are integers (Lemma \ref{fixedpts}).
\end{example}

\begin{thm}\label{basisphypo}
	A basis of $A(\z_p[G],\perm)$ is given by $\{\one \up^G_H\} _{ H \in \hyp_p(G) }$.
\end{thm}
\begin{proof}
	We first check linear independence. By Theorem \ref{conlon}, we need only check linear independence after taking species. By Lemma \ref{fixedpts}, the species $t_{(P,g)}(\one\up^G_H)$ is zero whenever no conjugate of $H$ contains $\la P,g\ra$, whilst $t_{(P,g)}(\one\up^G_{\la P,g\ra })\neq 0$. After (non-uniquely) ordering the elements of $\hyp_p(G)$ by increasing size, linear independence is now clear.
	
	Now, let $i$ be such that $\la P, g\ra =\la P,g^i\ra$ and let $\sigma \in \Aut(\co/\q)$ raise the $|g|$\th roots of unity to the $i$\th power. Then, for any trivial source module $M$ over $\z_p$,
	\begin{align*}
t_{(P,g^i)}(M)&=\tr(g^i \mid \textrm{vertex $P$ summand of $M\down_{\la P,g^i\ra}$})\\
&=\tr(g^i \mid \textrm{vertex $P$ summand of $M\down_{\la P,g\ra}$})\\
&=\tr(g \mid \textrm{vertex $P$ summand of $M\down_{\la P,g\ra}$})^\sigma\\
&=t_{(P,g)}(M)^\sigma.
	\end{align*}
	But for permutation modules $M$, $t_{(P,g)}(M)$ is rational (Lemma \ref{fixedpts}), so $t_{(P,g)}(M)$ is constant on pairs $(P,g)$ generating the same $p$-hypo-elementary subgroup up to conjugacy. Therefore, $\dim _\q A(\z_p[G],\perm)\le \# \{\textrm{conjugacy classes of $p$-hypo-elementary groups}\}$.
\end{proof}
We used this in Theorem \ref{dimbrelp} to find a basis of the space of Brauer relations in characteristic $p$. Examining the proof of the theorem we find:
\begin{corol} \label{ratspecies}
	For any finite group $G$ and prime $p$, $A(\z_p[G],\perm) \subseteq A(\z_p[G],\triv)$ is precisely the subspace of elements whose species are all rational.
\end{corol}
This also follows from work of Fan Yun \cite{FanYun}.
\begin{lemma}\label{reductioncondtophypo}
	Let $G$ be a finite group and $p$ a prime. If $A(\z_p[H],\perm)=A(\zp[H],\triv)$ for all $H\in \hyp_p(G)$, then $A(\z_p[G],\perm)=A(\zp[G],\triv)$.
\end{lemma}
\begin{proof}
	Suppose that $a\in A(\zp[G],\triv)$ is such that, upon restriction to every $p$-hypo-elementary subgroup $H$, $a\down_H$ is a permutation module. Then $a\down_H$ has rational species (Lemma \ref{fixedpts}). But then $a$ itself must have rational species as species are defined via restriction to the $p$-hypo-elementary subgroups. Applying Corollary \ref{ratspecies} we find $a \in A(\zp[G],\perm)$.
\end{proof}
As a result, a group $G$ satisfies the conditions of Theorem \ref{main} if its $p$-hypo-elementary subgroups do.
\begin{notat}
	Let $(P,g),(P',g')$ define two species and set $n=|g|$. We say $(P,g)\sim_p (P',g')$ if there exists an element $h\in G$ such that $(P')^h=P$ and $(g')^h= g^i$ for some $i \in (\z/n\z)^\ti $ with $i \equiv 1$ $(\textrm{mod } \gcd(n,p-1))$.
\end{notat}
\begin{lemma}\label{formulafortriv}
	For any finite group $G$ and prime $p$,
	\begin{enumerate}[i)]
		\item $\dim (A(\z_p[G],\triv))=\# \left(\{\textrm{species } (P,g) \} /\sim_p \right)$,
		\item $\dim (\im(A(\z_p[G],\triv)\to A(\q_p[G])))=\# \left(\{\textrm{species } (P,g) \mid \la P, g\ra \textrm{ is cyclic} \} /\sim_p \right)$.
	\end{enumerate}
\end{lemma}
\begin{proof}
	For any finite Galois extension $K/\q_p$, the action of $\gal(K/\q_p)$ on lattices respects decompositions into vertices. As a result, $t_{(P,g)}(-)=t_{(P,g^i)}(-)$ as functions on $A(\z_p[G],\triv)$ for any integer $i$ such that $(-)^i$ is an automorphism of $\la g\ra$ which acts trivially on the subgroup of order $m=\gcd(n,p-1)$, i.e.\ whenever $(P,g)\sim_p (P,g^i)$. Together with Theorem \ref{conlon}, this demonstrates the upper bound on $\dim A(\z_p[G],\triv)$. For the lower bound, use that the Green correspondence provides a distinct indecomposable trivial source module of vertex $P$ for every projective indecomposable $\zp[N_G(P)/P]$-lattice (see e.g.\ \cite[Thm.\ 3.12.2]{Benson95}), of which there are $\#(\{\textrm{species }(Q,h)\mid Q=P   \}/\sim_p)$.
	
	We now show {\it ii)}. The dimension of $A(\q_p[G])$ is equal to the number of distinct ring homomorphisms $\tr(g\mid -)\colon A(\q_p[G])\to\co$. The dimension of $\im (A(\z_p[G],\triv)\to A(\q_p[G]))$ is then the number of species up to $\sim_p$ which are of the form $\tr(g \mid -)$. By Example \ref{trasspecies}, this is equal to $\# \left(\{\textrm{species } (P,g) \mid \la P, g\ra \textrm{ is cyclic} \} /\sim_p \right)$.
\end{proof}
\begin{lemma}\label{condeqcond}
	The following are equivalent
	\begin{enumerate}[i)]
		\item $A(\z_p[G],\perm)=A(\zp[G],\triv)$,
		\item the species of all trivial source $\zp[G]$-lattices are rational,
		\item there is an equality
\begin{align*}\dim(A(\z_p[G],\perm)) =  &\dim(A({\z_p}[G],\triv)) \\ &-\dim (\im (A(\z_p[G],\triv)\to A(\q_p[G]))\\&  + \dim (A(\q[G])) ,\end{align*}
		\item there is an equality \begin{align*}\#\left(\{\textrm{$p$-hypo-elementary subgroups}\}/\sim \right) =  &\ \#\left(\{\textrm{species }(P,g) \}/\sim_p \right)\\ &-\# \left(\{\textrm{species } (P,g) \mid \la P, g\ra \textrm{ is cyclic} \} /\sim_p \right)\\&  + \#\left(\{\textrm{cyclic subgroups} \}/\sim\right) ,\end{align*}
		where $\sim$ denotes up to conjugacy.
			\end{enumerate}
\end{lemma}
\begin{proof}
	The equivalence {\it i)} $\iff$ {\it ii)} is Corollary \ref{ratspecies}. For {\it i)} $\iff$ {\it iii)}, use that a trivial source $\z_p[G]$-module is defined over $\zp$ if and only if it has rational character \cite[Prop.\ 5.7]{ReinerSurvey}, together with the fact that $A(\zp[G],\triv)\to A(\q[G])$ is surjective by Artin's induction theorem. For {\it iii)} $\iff$ {\it iv)}, combine Lemma \ref{formulafortriv}, Theorem \ref{basisphypo} and Artin's induction theorem. 
\end{proof}

 We conclude by giving examples of groups satisfying the condition $A(\z_p[G],\perm)=$ \linebreak $A(\zp[G],\triv)$:
\begin{example}\label{abelian}
	If $G$ is abelian with cyclic Sylow $p$-subgroups, then all $p$-hypo-elementary subgroups are cyclic and so, by Lemma \ref{condeqcond} {\it iv)}, $A(\z_p[G],\perm) =A(\zp[G],\triv)$ (note, there may be $\zp[G]$-lattices $M$ for which $(M\ot \z_p)_\triv$ does not lie in $A(\z_p[G],\perm)$, see the example of $C_3\ti C_4$ in \cite[Ch.\ 2]{thesis}). When applying Theorem \ref{main} for such $G$, the fact that there are no non-cyclic $p$-hypo-elementary subgroups makes the data of {\it ii)} empty.
\end{example}
\begin{example}\label{dihedral}
	Let $p$ be odd and $G=D_{2q}$ be the dihedral group of order $2q$ for any $q\ge 1$. Recall, that we need only check that the condition for all $p$-hypo-elementary subgroups. The only possible $p$-hypo-elementary subgroups of $G$ are either cyclic, in which case they are covered by the previous example, or of the form $D_{2p^r}$ for some $r\ge 1$. In that case, species up to $\sim_p$ are in bijection with subgroups and so Lemma \ref{condeqcond} {\it iv)} holds. 
\end{example}
\begin{example}
	If $(|G|,p-1)=1$, then all species of trivial source $\z_p[G]$-lattices are rational and so by Corollary \ref{ratspecies}, $A(\z_p[G],\perm)=A(\z_p[G],\triv)$. In particular, when $p=2$ all groups with cyclic Sylow $2$-subgroups satisfy the conditions of Theorem \ref{main}.
\end{example}
\section{Examples}\label{examples}
\subsection{$D_{2p}$}
	Let $G=D_{2p}=C_p\rti C_2$ be the dihedral group of order $2p$ for $p$ an odd prime. Then, $G$ satisfies the conditions of Theorem \ref{main} for all primes $\ell$, by Example \ref{dihedral}. Since all complex irreducible representations of $D_{2p}$ are defined over $\re$, all $\z_\ell[G]$-lattices are rationally self-dual for any $\ell$.

In this section, we explicate how Theorem \ref{main} distinguishes $\z_\ell[D_{2p}]$-lattices. Broadly, there are three different cases, when $\ell=2,p$ or when $\ell$ is coprime to the order of the group. In the latter case character theory applies and we shall for simplicity additionally assume that $\ell$ is chosen so that all $\q_\ell[G]$-representations are defined over $\q[G]$. For example, this is the case if $\ell$ is a primitive element modulo $p$. 

\label{D2pfullex}

It is only possible to go into such detail for groups of the form $D_{2p}$ as they are one of the few families of groups for which the isomorphism classes of all indecomposable $\z_\ell[G]$-lattices have been classified for all $\ell$ dividing $|G|$. 

By assumption all $\q_\ell[G]$-representations are defined over $\q$, so by \cite[Prop.\ 5.7]{ReinerSurvey} all $\z_\ell[G]$-lattices are defined over $\z_{(\ell)}[G]$. In Example \ref{dihedral}, we checked that $A(\z_\ell[G],\perm)=A(\z_{(\ell)}[G],\triv)$ for all primes $\ell$. So, as in Example \ref{tsD2p}, a basis of $A(\z_{\ell}[G],\perm)=A(\z_{(\ell)}[G],\triv)$ is given by
\[ S=\begin{cases}\one_{\{1\}}\up^G,\one_{C_2}\up^G,\one_{C_p}\up^G & \ell \neq p\\ \one_{\{1\}}\up^G,\one_{C_2}\up^G,\one_{C_p}\up^G, \one_G & \ell = p
\end{cases},\]	
and this is also a basis of $A(\z_\ell[G],\triv)$. When $\ell \neq 2$ or $p$, all $\z_{\ell}[G]$-lattices are projective so $S$ forms a basis of $A(\z_{\ell}[G])$.

	We can exhaust the non-trivial source modules via Yakovlev's theorem (Thm.\ \ref{yak}). When $\ell=2$, as $N_{D_{2p}}(C_2)=C_2$, the Yakovlev diagram for a module $M$ simply consists of $H^1(C_2,M)$ as an abelian group. So any $\z_{2}[G]$-lattice $M$ for which $H^1(C_2,M)\cong\z/2\z$ will extend $S$ to a basis of $A(\z_{2}[G])$. The sign representation $\eps$, that is the non-trivial one dimensional irreducible lifted from $\z_{2}[D_{2p}/C_p]$, is one such module.
	
	When $\ell=p$, the Yakovlev diagram of a $\z_p[G]$-lattice $M$ consists of $H^1(C_p,M)$ as a $\F_p[D_{2p}/C_p]$-module. Since $\char (\F_p)\neq 2$, there are two irreducible $\F_p[D_{2p}/C_p]$-modules, both one dimensional, one with trivial action and one without. So any two lattices whose cohomology exhibits these modules will extend $S$ to a basis of $A(\z_{p}[G])$. If $\rho$ denotes the $(p-1)$-dimensional irreducible $\q_p[G]$-representation, then there are two non-isomorphic $\z_{p}[G]$-sublattices $A,A'$ contained in $\rho$, with $H^1(C_p,A),H^1(C_p,A') \cong \z/p\z$ as abelian groups, but the former having non-trivial $D_{2p}/C_p$ action and $D_{2p}/C_p$ acting trivially on the latter. These modules are explicitly constructed in \cite{Lee64}. 
	
	In conclusion, 
	\[ \dim_\q(A(\z_{\ell}[G]))=  \begin{cases} 3 & \ell \neq 2,p\\
	4& \ell = 2\\
	6 & \ell = p
	\end{cases}.\]	 
	with a basis $S'$ given by
	\[ S'=\begin{cases}\one_{\{1\}}\up^G,\one_{C_2}\up^G,\one_{C_p}\up^G & \ell \neq 2,p\\
	\one_{\{1\}}\up^G,\one_{C_2}\up^G,\one_{C_p}\up^G,\eps & \ell = 2\\
	\one_{\{1\}}\up^G,\one_{C_2}\up^G,\one_{C_p}\up^G, \one_G, A,A' & \ell = p
	\end{cases}.\]	
Denote the extension of scalars map $A(\z_{\ell}[G])\to A(\q_\ell[G])$ by $a$, and the map $A(\z_{\ell}[G])\to \bigoplus_{H \in \nchyp_\ell(G)} \q$ which assigns a lattice $M$ the vector $(v_p(C_{\theta_H}(M)))_{H\in \nchyp_\ell(G)}$ by $b$. Then, Theorem \ref{main} states that $a \oplus b\oplus \textrm{Yak}$ is injective (and so an isomorphism). The matrix representing $a \oplus b\oplus \textrm{Yak}$ is given by:  
								\begin{figure}[H]
		\begin{tikzpicture}[
		every left delimiter/.style={xshift=.75em},
		every right delimiter/.style={xshift=-.75em},
		]
		\matrix [ matrix of math nodes,left delimiter={( },right delimiter={ )},row sep=0.1cm,column sep=0.1cm] (U) { 
	1 &1 &1\\
 1 & \ \ \  \  0 \ \ \ \ &1\\
2 & 1&0\\
		};

		\node[white] at (-5.2,0) {{.}};

		\node[left=10pt  of U-1-1]  {$\one$};
		\node[left=12pt  of U-2-1]  {$\eps$};  		
				\node[left=12pt  of U-3-1]  {$\rho$}; 
		
		\node[above=.5pt  of U-1-1]  {$\one \up^G_{\{1\}}$}; 	
		\node[above=1pt  of U-1-2]  {$\one \up^G_{C_2}$}; 
		\node[above=0pt  of U-1-3]  {$\one \up^G_{C_p}$}; 
		
		\node at (4.2,0) {if $\ell\neq 2 ,p$ ,};
		\end{tikzpicture}
	\end{figure}\noindent  
							\begin{figure}[H]
		\begin{tikzpicture}[
		every left delimiter/.style={xshift=.75em},
		every right delimiter/.style={xshift=-.75em},
		]
		\matrix [ matrix of math nodes,left delimiter={( },right delimiter={ )},row sep=0.1cm,column sep=0.4cm] (U) { 
1 &   1   & 1    & 0\\
1 & 0&1&1\\
2 & 1&0&0\\
 0 &0 &0 &1\\
		};

		\node[white] at (-5,0) {{.}};
		
		\draw (.85,-1.1) -- (.85,1.1);
		\draw (1.4,-.55) -- (-1.4,-0.55);

\node[left=10pt  of U-1-1]  {$\one$};
\node[left=12pt  of U-2-1]  {$\eps$};  		
\node[left=12pt  of U-3-1]  {$\rho$}; 
		\node[left=12pt  of U-4-1]  {$\text{\rm Yak}(-)$};

		\node[above=0pt  of U-1-1]  {$\one \up^G_{\{1\}}$}; 	
		\node[above=0pt  of U-1-2]  {$\one \up^G_{C_2}$}; 
		\node[above=0pt  of U-1-3]  {$\one \up^G_{C_p}$}; 
		\node[above=3pt  of U-1-4]  {$\, \eps \, $}; 
		
		\node at (4,0) {if $\ell=2$,};
		\end{tikzpicture}
	\end{figure}\noindent 
						\begin{figure}[H]
		\begin{tikzpicture}[
		every left delimiter/.style={xshift=.75em},
		every right delimiter/.style={xshift=-.75em},
		]
		\matrix [ matrix of math nodes,left delimiter={( },right delimiter={ )},row sep=0.1cm,column sep=0.4cm] (U) { 
1 &1 &1& 1& 0&0\\
 1 & 0&1&0&0&0\\
2 & 1&0&0&1&1\\
0 &0 &0 &\! -1/2 \!&\! 1/2\!&\!-1/2\\
0&0&0&0&1&0\\
0&0&0&0&0&1\\
		};

		\node at (1,-0.5) {};
		
		\draw (-.55,-1.7) -- (-.55,1.7);
		\draw (.6,-1.7) -- (0.6,1.7);
		\draw (-2.85,0.1) -- (2.85,0.1);
		\draw (-2.85,-0.55) -- (2.85,-0.55);

\node[left=10pt  of U-1-1]  {$\one$};
\node[left=12pt  of U-2-1]  {$\eps$};  		
\node[left=12pt  of U-3-1]  {$\rho$}; 
		\node[left=10pt  of U-4-1]  {$v_p{(C_{\theta_G}(-))} $}; 
		\node at (-4,-1.3)  {$\text{\rm Yak}(-)$};

		\node[above=0pt  of U-1-1]  {$\one \up^G_{\{1\}}$}; 	
		\node[above=.5pt  of U-1-2]  {$\one \up^G_{C_2}$}; 
		\node[above=0pt  of U-1-3]  {$\one \up^G_{C_p}$}; 
		\node[above=3pt  of U-1-4]  {$\one_G$}; 
		\node[above=3pt  of U-1-5]  {$A$}; 
		\node[above=3pt  of U-1-6]  {$A'$}; 
		
		\node at (4,0) {if $\ell=p$.};
		\end{tikzpicture}
	\end{figure}\noindent
	Here, we take the basis $\one, \eps , \rho$ of $\q_\ell[G]$, where $\one ,\eps$ are the trivial and non-trivial one dimensional irreducibles and $\rho$ the $(p-1)$ dimensional irreducible. When $\ell=2$, the basis of $\mathcal{C}$ is taken to be $\z/2\z$, and when $\ell=p$, the basis is given by $\z/p\z$ with both its non-trivial and trivial $C_2\cong D_{2p}/C_p$-actions respectively. The calculations of $v_p(C_{\theta_G}(A)),v_p(C_{\theta}(A'))$ can be found in \cite[Thm.\ 4.4]{BartelDih}.

\begin{remark}\label{67}
	For $p\le 67$, a $\z[D_{2p}]$-lattice is determined by its localisation at the primes $2,p$ (see \cite[Ex.\ 6.3]{BartelDih}). So, by applying Theorem \ref{main} at both primes we obtain a finite list of data which specifies the isomorphism class of an arbitrary $\z[D_{2p}]$-lattice.
\end{remark}
\begin{remark}
	The above matrices can be seen to be block upper triangular. This was touched on in the proof of Lemma \ref{tsdet} and is a general phenomenon.
\end{remark}
\subsection{Groups with degenerate permutation pairing}
\label{counterexamples}
\begin{example}\label{C3C3S3}
	Let $G=C_3\ti C_3 \ti S_3$, a $3$-hypo-elementary group. Up to conjugacy $G$ has 17 subgroups, but the permutation pairing of Construction \ref{permpairingconstr} is degenerate and has rank 16.
\end{example}
In this section, we define a canonical Brauer relation $\theta_{\Sigma,G}$ which is non-zero for any non-cyclic group $G$. When $G=C_3 \ti C_3 \ti S_3$, then $\theta_{\Sigma,G}$ generates the kernel of the permutation pairing.
\begin{notat}\begin{itemize}
		\item 	Let $G$ be any finite group and $\Sigma$ denote the set of all subgroups of $G$, which is partially ordered with respect to containment. Let $\mu_\Sigma \colon \Sigma \to \z$ denote the M\"obius function on $\Sigma$, i.e.\ the unique function for which
		\[   \mu_\Sigma(G)  =  1  \]
		and
		\[ \sum_{\substack{H'\ge H}} \mu_\Sigma(H')=0    \]
		for all $H\neq G$.
		\item  Set 
		\[\theta_{\Sigma,G}= \sum_{H\in \Sigma} \frac{\mu_\Sigma(H)}{|G: H|}[H]  \in B(G). \]
		\item For an element $\theta\in B(G)$ and $K\le G$, let $\theta^K$ denote the number of fixed points of $\theta$ under $K$, i.e.\ if $\theta = \sum_H \alpha_H[H]$ then $\theta^K=\sum \alpha_H \#([H]^K)$.
	\end{itemize}
\end{notat}

\begin{lemma} \label{mobiusfixedpts}
	For any $K\le G$, $(\theta_{\Sigma,G})^K=\sum_{ H \ge K} \mu_\Sigma(H)$.
\end{lemma}
\begin{proof}
	Since both $\#[H]^K$ and $\mu_\Sigma(H)$ are constant under replacing $H$ with a conjugate, we have
	\begin{align*}
	\sum_{H \le G}  \frac{\mu_\Sigma(H)}{|G : H|}\#([H]^K)&= \sum_{H\le_GG} \frac{|H| \mu_\Sigma(H)}{|N_G(H)|}\#([H]^K)\\
	&=\sum_{H\le_GG} \frac{|H| \mu_\Sigma(H)}{|N_G(H)|} |\{g \in G/H \mid K^g\le H  \}|\\
	&=\sum_{H\le_GG} \frac{|H| \mu_\Sigma(H)}{|N_G(H)|} |\{g \in G/H \mid K\le H^g  \}|\\
	&=\sum_{H\le_GG} \frac{|H| \mu_\Sigma(H)}{|N_G(H)|}\cdot \frac{|N_G(H)|}{|H|} |\{g \in G/N_G(H) \mid K\le H^g  \}|\\
	&= \sum_{ H \ge K } \mu_\Sigma(H).
	\end{align*}
\end{proof}
\begin{corol}\label{mobius result}
	For any finite group $G$,
	\begin{enumerate}[i)]
		\item $\theta_{\Sigma,G}$ is a Brauer relation in characteristic zero if and only if $G$ is non-cyclic,
		\item $\theta_{\Sigma,G}$ is a Brauer relation in characteristic $p$ if and only if $G$ is non-$p$-hypo-elementary,
		\item $\theta_{\Sigma,G}\down_H$ is zero for all proper subgroups $H$.
	\end{enumerate}
\end{corol}
\begin{proof}
	We first check {\it i)}. An element of $B(G)$ is a relation in characteristic zero if and only if the number of its fixed points under all cyclic subgroups is zero (see for example the proof of \cite[Thm.\ 5.6.1]{Benson95}). By the lemma, $\theta_{\Sigma,G}$ is a relation in characteristic zero if and only if 
	\[  \sum_{ H \ge C }\mu_\Sigma(H)=0\]
	for all cyclic subgroups $C$. By the definition of $\mu_{\Sigma}$, this is true if and only if $G$ is not itself cyclic.
	
	The argument for {\it ii)} is identical instead using that elements of $B(G)$ are relations in characteristic $p$ if and only if the number of fixed points under all $p$-hypo-elementary subgroups is zero (repeat the proof of \cite[Thm.\ 5.6.1]{Benson95} but using species and Lemma \ref{fixedpts}).
	
	For {\it iii)}, simply note that an element of $B(H)$ is zero if and only if its fixed points under all subgroups is zero (proven analogously to the previous cases). But $(\theta_{\Sigma,G}\down_H)^K=(\theta_{\Sigma,G})^K$, which by the lemma vanishes for all proper subgroups $K<G$.
\end{proof}
\begin{remark}
	By Lemma \ref{regcontoolkit}\,{\it iv)}, $\theta_{\Sigma,G}$ automatically vanishes on all permutation modules other than possibly the trivial representation. In Example \ref{C3C3S3}, $C_3\ti C_3\ti S_3$ is a group for which in addition $\theta_{\Sigma,G}(\one_G)=1$.
\end{remark}
\begin{example}
	Let $G=(C_p\ti C_p)\rti C_q$ with $p,q$ odd primes and $p=2q+1$ and where $C_q$ acts diagonally on $C_p\ti C_p$. Write $\alpha_H= \mu_{\Sigma}(H)/|G :H|$ so that $\theta_{\Sigma,G}=\sum_{H\le G}\alpha_H[H]$. Then, the $\alpha_H$ for each conjugacy class are given in the following table:
	\begin{figure}[H]
		\begin{tabular}{ c | c c c c c c c c c c c}
			&$1$ & $C_q $& $C_p$ & $C_p$ &$C_p$ &$C_p$& $C_p \rti C_q $& $C_p \rti C_q$ & $C_p \ti C_p $& $G$\\
			\hline
			\#conjugates & $1$ &$p^2$ & $1$&$1$&$q$&$q$&$p$ &$p$ &$1$&$1$\\
			$\mu_{\Sigma}(H)$ & $-p^2$ & $1 $& $p$&$ p$ & $0$&$0$ &$-1$&$-1$&$-1$&$1$\\
			$\alpha_H$ & $-1/q$ & $1/p^2$ & $1/q$ & $1/q$ & $0$ & $0$ & $-1/p$& $-1/p$ & $-1/q$ & $1$ 
		\end{tabular}
	\end{figure}
	And we find $\theta_{\Sigma,G}(\one_G)=1$. Thus, the relation $\theta_{\Sigma,G}$ is trivial on all permutation representations. As $G$ is $p$-hypo-elementary, $\theta_{\Sigma,G}$ is not a $p$-relation and the permutation pairing of Construction \ref{permpairingconstr} is degenerate.
\end{example}
\begin{remark}
It is not clear to the author if, for any of the above groups, there exists a lattice $M$ for which $\theta_{\Sigma,G}(M)\neq 1$. Necessarily, such an $M$ must not be induced from a proper subgroup (Lemma \ref{regcontoolkit}). For $G=C_3\ti C_3\ti S_3$, $A(\z_{3}[G],\perm)=A(\z_{(3)}[G],\triv)=A(\z_{3}[G],\triv)$ since $G$ has the same number of species as conjugacy classes of $p$-hypo-elementary subgroups (Lemma \ref{condeqcond} {\it iv)}). Thus, such an $M$ would also have to not have trivial source.
\end{remark}

\bibliographystyle{amsalphainitials}

\renewcommand{\refname}{References}
\bibliography{MyCollection}
\vspace{0.7cm}
\textsc{Mathematics Institute, Zeeman Building, University of Warwick, Coventry
	CV4 7AL, UK}

\noindent \emph{E-mail address:}	\texttt{alex.torzewski@gmail.com}

\end{document}